\documentclass[10pt, a4paper]{amsart}
\usepackage{amssymb}
\usepackage{amsmath}
\usepackage{amssymb}
\usepackage{amsthm}
\usepackage{bbm}
\usepackage{bm}
\usepackage{mathrsfs}
\usepackage[usenames,dvipsnames,svgnames,table]{xcolor}
\allowdisplaybreaks
\usepackage{enumitem}
\setenumerate{label={\rm (\alph{*})}}
\usepackage[colorlinks=true,citecolor=Mahogany,urlcolor=blue,
linkcolor=RoyalBlue]{hyperref}

\usepackage{graphicx}
\usepackage{tikz}

\usepackage{geometry}
\numberwithin{equation}{section}

\newtheorem{theorem}{Theorem}[section]
\newtheorem{lemma}[theorem]{Lemma}
\newtheorem{proposition}[theorem]{Proposition}

\newtheorem{counterexample}[theorem]{Counterexample}

\newtheorem{definition}[theorem]{Definition}

\newtheorem{conjecture}[theorem]{Open Problem}

\numberwithin{equation}{section}

\usepackage{booktabs}

\normalsize

\def\XXint#1#2#3{{\setbox0=\hbox{$#1{#2#3}{\int}$ }
\vcenter{\hbox{$#2#3$ }}\kern-.6\wd0}}

\newcommand{\bd}{\operatorname{BD}}
\newcommand{\bv}{\operatorname{BV}}

\newcommand{\di}{\operatorname{div}}

\newcommand{\dif}{\operatorname{d}\!}

\newcommand{\spt}{\operatorname{spt}}

\newcommand{\tr}{\operatorname{Tr}}

\newcommand{\R}{\mathbb{R}}
\newcommand{\A}{\mathbb{A}}
\newcommand{\C}{\mathbb{C}}

\newcommand{\locc}{\operatorname{loc}}

\newcommand{\trace}{\operatorname{Tr}}

\newcommand{\ball}{B}

\newcommand{\spano}{\operatorname{span}}

\newcommand{\inte}{\operatorname{int}}

\newcommand{\sobo}{\operatorname{W}}
\newcommand{\Vsob}{\operatorname{V}}

\newcommand{\lebe}{\operatorname{L}}

\newcommand{\hold}{\operatorname{C}}

\newcommand{\D}{D}
\newcommand{\curl}{\operatorname{curl}}
\renewcommand{\leq}{\leqslant}

\newcommand{\imag}{\operatorname{i}}

\newcommand{\besov}{\operatorname{B}}

\renewcommand{\di}{\operatorname{div}}

\newcommand{\diam}{\operatorname{diam}}

\newcommand{\lin}{\mathscr{L}}

\usepackage{tikz-cd}
\usepackage{combelow}
\begin{document}
\title[Embeddings of $\sobo^{\A,1}$]{Embeddings for $\A$--weakly differentiable functions on domains}
\author[F. Gmeineder]{Franz Gmeineder}
\author[B. Rai\cb{t}\u{a}]{Bogdan Rai\cb{t}\u{a}}
%\address[F.~Gmeineder]{Universit\"{a}t Bonn, Mathematisches Institut, Endenicher Allee 60, Bonn, Germany. \normalfont{\textit{Email}: \texttt{fgmeined@math.uni-bonn.de}}}
%\address[B.~Rai\cb{t}\u{a}]{University of Warwick, Zeeman Building, Coventry, CV4 7HP, United Kingdom. \normalfont{\textit{Email}: \texttt{bogdan.raita@warwick.ac.uk}}}
\thanks{\emph{Author's address}: F.~Gmeineder, Bonn, Mathematisches Institut, Endenicher Allee 60, Bonn, Germany. \normalfont{\textit{Email}: \texttt{fgmeined@math.uni-bonn.de}};\\
	\indent\emph{Corresponding author}: B.~Rai\cb{t}\u{a}. \emph{Address}: University of Warwick, Zeeman Building, Coventry, CV4 7HP, United Kingdom. \normalfont{\textit{Email}: \texttt{bogdan.raita@warwick.ac.uk}}}
\maketitle

 \begin{abstract}
 	We prove that the inhomogeneous estimate of vector fields on balls in $\mathbb{R}^n$
 	$$
 	\left(\int_{B}|D^{k-1}u|^{n/(n-1)}\mathrm{d} x\right)^{(n-1)/n}\leq c\left(\int_{B}|\mathbb{A} u|+|u|\mathrm{d} x\right)\text{ for  all }u\in\mathrm{C}^\infty(\bar B,\mathbb{R}^N)
 	$$
 	holds if and only if the linear, constant coefficient differential operator $\mathbb{A}$ of order $k$ has finite dimensional null-space (FDN). This generalizes the Gagliardo-Nirenberg-Sobolev inequality on domains and provides the local version of the analogous homogeneous embedding in full-space
 	$$
 	\left(\int_{\mathbb{R}^n}|D^{k-1}u|^{n/(n-1)}\mathrm{d} x\right)^{(n-1)/n}\leq c\int_{\mathbb{R}^n}|\mathbb{A} u|\mathrm{d} x\qquad\text{ for  all }u\in\mathrm{C}^\infty_c(\mathbb{R}^n,\mathbb{R}^N),
 	$$
 	proved by Van Schaftingen precisely for elliptic and cancelling (EC) operators, building on fundamental $\mathrm{L}^1$-estimates from the works of Bourgain and Brezis. We prove that FDN strictly implies EC and discuss the contrast between homogeneous and inhomogeneous estimates on both algebraic and analytic level.
 \end{abstract}
\section{Introduction}\label{sec:intro}
A known principle in harmonic analysis is that strong--type $\lebe^1$--estimates are notoriously delicate to obtain. For example, singular integrals and Riesz potentials are only bounded from $\lebe^1$ into a weak--type space, which contrasts the case of $\lebe^p$--spaces, $p>1$. To note that these $\lebe^1$--estimates of weak--type are sharp, one simply tests the inequalities with an approximation of the identity.

Historically, this discrepancy can be observed already from the original proof of the Sobolev inequality for $1<p<n$ and $p^*=\frac{np}{n-p}$,
\begin{align}\label{eq:sobolev}
\|u\|_{\lebe^{p^*}(\R^n)}\leq c\|\D u\|_{\lebe^{p}(\R^n)}
\end{align}
for $u\in\hold^\infty_c(\R^n)$. The technique of proof in \cite{Sobolev} resembles proving boundedness of the Riesz potential $I_1$ between $\lebe^p$ and $\lebe^{np/(n-p)}$, but it is in no way adaptable to the $p=1$ case. It was much later that \textsc{Gagliardo} \cite{Gag} and \textsc{Nirenberg} \cite{Nir} independently showed with new methods that \eqref{eq:sobolev} holds also for $p=1$, in particular showing that the specific vectorial structure of the gradient operator allows for a strong--type estimate, despite unboundedness of $I_1$ between $\lebe^1$ and $\lebe^{1^*}$.

Coupled with the fact attributable to \textsc{Calder\'on} and \textsc{Zygmund} \cite{CZ} that 
\begin{align}\label{eq:CZ}
\|\D^k u\|_{\lebe^p(\R^n)}\leq c\|\A u\|_{\lebe^p(\R^n)}
\end{align}
for $k$--th order elliptic operators $\A$ and $1<p<\infty$, it seems plausible that the vectorial structure of some operators $\A$ may also compensate for unboundedness of singular integrals on $\lebe^1$. This fact was disproved by \textsc{Ornstein} in \cite{Ornstein} (see also \cite{KirKri}), where it is shown that the estimate \eqref{eq:CZ} holds for $p=1$ only if $|\D^k u|\leq c|\A u|$ pointwisely for all $u\in\hold^\infty_c$.

However, it remained possible that strong--type $\lebe^1$--estimates for lower order derivatives can be deduced. Indeed, it was proved by \textsc{Strauss} in \cite{Strauss} that \eqref{eq:sobolev} holds for $p=1$ when the $\lebe^1$--norm of $\D u$ is replaced by the weaker quantity $\|\mathcal{E}u\|_{\lebe^1}$. Here $\mathcal{E}u=\frac{1}{2}(\D u+(\D u)^t)$ denotes the symmetrized gradient of $u$. More recently, it was proved by \textsc{Bourgain} and \textsc{Brezis} in the pioneering works \cite{BB04,BB07} that, for Poisson's equation in $\R^n$, $n\geq2$,
 $$\Delta u=f,$$
the surprising strong $\lebe^1$--estimate
\begin{align*}
\|\D u\|_{\lebe^{1^*}}\leq c\|f\|_{\lebe^1}
\end{align*}
holds provided that  $f\in\hold^\infty_c$ is divergence--free. This and substantial contributions in \cite{BBM1,BBM2,BBM3,BB02,BBCR,BB04,BB07,VS-1,VS0,VS1,VS4}
lead to the remarkable characterization by \textsc{Van Schaftingen} \cite{VS} of all $k$--homogeneous linear differential operators $\A$ such that
\begin{align}\label{eq:VS}
\|\D^{k-1}u\|_{\lebe^{1^*}(\R^n)}\leq c\|\A u\|_{\lebe^1(\R^n)}
\end{align}
for all $u\in\hold^\infty_c$. The class of operators $\A$ for which \eqref{eq:VS} holds is that of \emph{elliptic} and \emph{cancelling} operators (EC). Both these assumptions are defined in terms of the symbol map of the operator $\A$, the definition of which we now recall. We will represent $k$--homogeneous linear differential operators with constant coefficients on $\R^n$ from $V$ to $W$ as 
\begin{align}\label{eq:A}
\A u=\sum_{|\alpha|=k} A_\alpha\partial^\alpha u,\qquad u\colon\R^{n}\to V,
\end{align}
where $A_\alpha\in\lin(V,W)$ are fixed linear mappings between two finite dimensional normed real vector spaces $V$ and $W$. The symbol map is defined as
\begin{align*}
\A[\cdot]\colon \R^{n}\to \mathscr{L}(V,W),\;\;\;\A[\xi] v=\sum_{|\alpha|=k} \xi^\alpha A_\alpha v,
\end{align*}
defined for $\xi\in\R^n$, $v\in V$. Algebraically, (overdetermined) \emph{ellipticity} is defined by injectivity of the symbol map $\A[\xi]$ for all non--zero $\xi\in\R^n$, whereas \emph{cancellation}, introduced in \cite[Def.~1.2]{VS}, is defined by
\begin{align}\label{eq:canc}
\bigcap_{\xi\in\R^n\setminus\{0\}}\mathrm{im\,}\A[\xi]=\{0\}.
\end{align}
We will use the short--hand EC for operators that are elliptic and cancelling.

Analytically, ellipticity is equivalent to the classical estimate \eqref{eq:CZ}. Surprisingly and interestingly, cancellation is equivalent to non--admissibility for \eqref{eq:VS} of approximations of the identity, in the sense that if  $\A u_\varepsilon=\varphi_\varepsilon w$ for standard mollifiers $\varphi_\varepsilon\in\hold^\infty_c(\R^n)$ and $w\in W$, then $w\in\mathrm{im\,}\A[\xi]$ for all $\xi\neq0$.

An overarching overview of these and other recent developments on $\lebe^1$--estimates can be found in \cite{VS3}, where it is also asked in Open~Problem~3 whether, under suitable complementing boundary conditions, one can develop global strong--type estimates on domains. It is implicitly conjectured that estimates on smooth domains $\Omega\subset\R^n$ such as
\begin{align*}
\|\D^{k-1}u\|_{\lebe^{1^*}(\Omega)}\leq\|\A u\|_{\lebe^1(\Omega)}+\|u\|_{\lebe^1(\Omega)},
\end{align*}
provided that $\A$ is EC and $u\in\hold^\infty(\bar\Omega,V)$ satisfy $\mathbb{B}_j u=0$ on $\partial\Omega$, where $\mathbb{B}_j$ is a (finite collection of) linear differential operator(s) defined on $\partial\Omega$ that satisfy the Lopatinski\u{\i}--Shapiro Complementing Conditions. Such a result would provide a reasonable analogue of the results in \cite{Lopa,ADN1,ADN2,HormanderBdry} to the case $p=1$, in spite of Ornstein's Non--inequality.

The aim of this paper is to confirm this expectation in the case when $\mathbb{B}_j\equiv0$ (``no boundary condition'') and $\Omega$ is a ball (whereas \textsc{Van Schaftingen}'s result \cite[Thm.~1.3]{VS} essentially deals with the antipodal case when $\mathbb{B}_j=\partial_\nu^j$, $j=0\ldots k-1$, i.e., ``all boundary conditions''). We emphasize that in the present situation the geometry of $\partial\Omega$ is \textbf{not} the foremost problem, as is the extendibility of functions $u\colon\Omega\to V$ to some $v\colon\R^{n}\to V$ while ensuring that $\A v\in\lebe^{1}(\R^{n},V)$ boundedly. In fact, as we shall see below, this property fails for a large class of EC operators. For more generality, the reader can easily use our arguments to cover the case when $\Omega$ is a bounded Lipschitz domain. %With more work, it is feasible to also cover $(\varepsilon,\delta)$--domains.

The complementing conditions for smooth domains for an elliptic operator $\A$ and identically zero boundary conditions \cite[Def.~20.1.1]{HormanderIII} can be rephrased as
\begin{align}\label{eq:C_ell}
\A[\xi]\in\lin(V+\imag V,W+\imag W)\qquad\text{ is injective for all }\xi\in\C^n\setminus\{0\}.
\end{align}
This condition, referred to as \emph{$\C$--ellipticity} in \cite{BDG}, and attributable to \textsc{Aronszajn} \cite{Aronszajn} (at least in the case of scalar--valued maps), was introduced in \cite{Smith0,Smith} to characterize operators $\A$ such that the local variant of \eqref{eq:CZ} holds, i.e.,
\begin{align}\label{eq:korn_dom}
\|\D^k u\|_{\lebe^p(\ball)}\leq c\left(\|\A u\|_{\lebe^p(\ball)}+\|u\|_{\lebe^p(\ball)}\right)
\end{align}
holds for $u\in\hold^\infty(\bar\ball,V)$ (here $1<p<\infty$). By Ornstein's Non--inequality, no such estimate is possible for $p=1$, but, inspired by \textsc{Van Schaftingen}'s Theorem and Open Problem, we will prove that $\C$--ellipticity of $\A$ is equivalent to the estimate
\begin{align}\label{eq:main_emb}
\|\D^{k-1} u\|_{\lebe^{n/(n-1)}(\ball)}\leq c\left(\|\A u\|_{\lebe^1(\ball)}+\|u\|_{\lebe^1(\ball)}\right)
\end{align}
for $u\in\hold^\infty_c(\bar\ball,V)$. In particular, we recover the Gagliardo--Nirenberg--Sobolev inequality on domains and the Korn--Sobolev inequality \cite[Prop.~1.2]{ST}, due to \textsc{Strang} and \textsc{Temam}.

To formally state our results, we prefer to use the functional framework of $\A$--weakly differentiable functions and define, in the spirit of \cite{Adolfo,BDG}, the space $\sobo^{\A,1}(\ball)$ as the space of $u\in\lebe^1(\ball,V)$ such that $\A u\in\lebe^1(\ball,W)$ with the obvious norm. In view of characterizing
operators $\A$ such that \eqref{eq:main_emb} holds, we firstly confine to operators of order one  for simplicity and state our main result in the following slightly more elaborate form:
\begin{theorem}\label{thm:main}
Let $\A$ be as in \eqref{eq:A}, $k=1$, $n>1$. The following are equivalent:
\begin{enumerate}
\item\label{it:main_a} $\A$ is $\C$--elliptic.
\item\label{it:main_b} $\sobo^{\A,1}(\ball)\hookrightarrow\lebe^{\frac{n}{n-1}}(\ball,V)$.
\item\label{it:main_c} $\sobo^{\A,1}(\ball)\hookrightarrow\lebe^{p}(\ball,V)$ for some $1<p\leq\frac{n}{n-1}$.
\item\label{it:main_d} $\sobo^{\A,1}(\ball)\hookrightarrow\hookrightarrow\lebe^{q}(\ball,V)$ for all $1\leq q<\frac{n}{n-1}$.
\item\label{it:main_e} $\sobo^{\A,1}(\ball)\hookrightarrow\hookrightarrow\lebe^{1}(\ball,V)$.
\end{enumerate}
\end{theorem}
This result manifests the following dichotomy: Either $\A$ is $\C$--elliptic (in which case
one retrieves the known Sobolev embeddings on domains for the gradient or symmetric
gradient), or $\A$ is not $\C$--elliptic. In the latter case, a map $u\in\sobo^{\A,1}(B)$ only belongs to $\lebe^1(B,V)$ (by definition), but \emph{no better} $\lebe^p(B,V)$--space. The techniques and auxiliary statements involved to arrive at the conclusion of the theorem play a role in a variety of related statements, see for instance the forthcoming work \cite{GR}. 

To prove Theorem~\ref{thm:main}, a crucial step is to show that $\C$--ellipticity implies cancellation, which is one of the main novelty of this paper. Another important step is to extend to full--space and employ \eqref{eq:VS}. The fact that the sub--critical embedding in  \ref{it:main_d} is compact generalizes the well--known result for $\bd$ (i.e., for $\A=\mathcal{E}$; see \cite{FS,ST,Suquet}) and is achieved by a careful application of the Riesz--Kolmogorov criterion.

%As to the proof, i
It is a priori far from obvious how can one connect the algebraic definitions of $\C$--ellipticity \eqref{eq:C_ell} and cancellation \eqref{eq:canc}. We, however, consider the analytic characterization of cancellation \cite[Prop.~6.1]{VS} and a consequence of \textsc{Smith}'s representation formulas \cite{Smith} (the kernel of a $\C$--elliptic operator is finite dimensional) to build a bridge in Lemma~\ref{lem:FDNimpliesEC}. In fact, it was already observed in \cite{BDG} for first order operators that $\C$--ellipticity of $\A$ is equivalent with 
\begin{align*}
\dim\{u\in\mathscr{D}^\prime(\R^n,V)\colon\A u=0\}<\infty.\tag{FDN}
\end{align*}
In Proposition~\ref{prop:FDNiffTypeC} we will give a short proof of the fact that $\mathbb{C}$--ellipticity and FDN are equivalent for operators of arbitrary order. Henceforth, we will thus use ``FDN'' and ``$\mathbb{C}$--elliptic'' interchangeably.

On the other hand, we also show that the implication of EC by $\C$--ellipticity is strict by considering the first order operator
\begin{align*}
\A u=\left(\partial_1u_1-\partial_2u_2,\,\partial_2u_1+\partial_1u_2,\,\partial_3u_1,\,\partial_3u_2\right)\qquad\text{for }u\colon\R^3\rightarrow\R^2.
\end{align*}
In particular, for this operator, there are maps in $\sobo^{\A,1}(\ball)$ that are locally $\frac{n}{n-1}$--integrable, but are \emph{not} $\lebe^p$--integrable up to the boundary for any $p>1$. Nor do they admit traces in $\lebe^1(\partial,\ball,V)$ by \cite[Thm.~4.18]{BDG}. We further expand on these points in Section~\ref{sec:EC>emb}.

In Section~\ref{sec:ECvsFDN} we will give a more comprehensive comparison of the notions of $\C$--ellipticity and EC, depending on $n$, $N=\dim V$, and the order $k$ of $\A$. This seems to be the first comparison between Lopatinski\u{\i}--Shapiro Complementing Conditions at the boundary and algebraic conditions characterizing interior $\lebe^1$--estimates.
\begin{proposition}\label{prop:intro_EC_vs_FDN}
	Let $\A$ be as in \eqref{eq:A}, $n>1$, and $N=\dim V$. If $\A$ is $\C$--elliptic, then $\A$ is elliptic and cancelling. The converse fails in general, except in the following cases:
	\begin{enumerate}
		\item $k=N=1$, when ellipticity implies $\C$--ellipticity;
		\item ($N=1$ and $k=2$) or ($N\geq 2$ and $k=1$), when EC implies $\C$--ellipticity.
	\end{enumerate}
\end{proposition}
We also summarize the content of Proposition~\ref{prop:intro_EC_vs_FDN} in a table, Figure~\ref{fig:table} below.
%\vspace{-10pt}
\begin{figure}[ht]\label{fig:table}
	\begin{tabular}{|c|c|}
	\hline
	&	\begin{tabular}{c|c}
		$N=1\qquad\quad$&$\qquad\quad N\geq2$
		\end{tabular}\\
	\hline
	$n=2$&
	\begin{tabular}{c|c}
		\begin{tabular}{l}
		$k=1$: \textcolor{NavyBlue}{E$\Rightarrow$FDN}\\
		\hline
		$k=2$: \textcolor{ForestGreen}{EC$\Rightarrow$FDN}\\
		\hline
		$k\geq3$: \textcolor{BrickRed}{EC$\not\Rightarrow$FDN}
		\end{tabular}&
		\begin{tabular}{l}
		$k=1$: \textcolor{ForestGreen}{EC$\Rightarrow$FDN}\\
		\hline
		$k\geq2$: \textcolor{BrickRed}{EC$\not\Rightarrow$FDN}\\
		\\
		\end{tabular}\\
	\end{tabular}\\
	\hline
	$n\geq3$&
	\begin{tabular}{c|c}
		\begin{tabular}{l}
		$k=1$: \textcolor{NavyBlue}{E$\Rightarrow$FDN}\\
		\hline
		$k\geq2$: \textcolor{BrickRed}{EC$\not\Rightarrow$FDN}
		\end{tabular}&
		\begin{tabular}{l}
		\textcolor{BrickRed}{EC$\not\Rightarrow$FDN}$\qquad\quad\hspace{0.4mm}$\\
		\end{tabular}
	\end{tabular}\\	
	\hline
	\end{tabular}%\\[5pt]
\caption{Comparison of the conditions E (elliptic), EC (elliptic and canceling), and FDN (finite dimensional null--space) for operators $\A$ as in \eqref{eq:A} acting on maps $u\colon\R^n\rightarrow V$, depending on the order $k$, and the dimensions $n$ and $N=\dim V$.}
\end{figure}

The matter of the extending $\sobo^{\A,1}(B)$--maps to $\R^n$ is technically very delicate, as we cannot formulate a simple proof as in \cite{Smith,Kalamajska94} due to lack of boundedness of singular integrals on $\lebe^1$ (cp. Lemma~\ref{lem:extp>1}). Instead, we resort to the involved technique introduced by \textsc{Jones} \cite{Jones}. From a conceptual perspective, this method crucially
relies on inverse estimates for polynomials and thereby underlines the need of the
FDN. The ideas for adapting this construction for FDN operators instead originates in \cite{BDG}, where a related result was proved for first order operators. Namely, it was shown that maps in $\sobo^{\A,1}(\ball)$ admit $\lebe^1$--traces on $\partial\ball$ if and only if $\A$ has FDN. In particular, their result and ours point in the direction that FDN is a suitable assumption for estimates near and on the boundary also in the case $p=1$.

In general,
we have the following result, which we believe to be of independent interest besides serving
as a crucial tool in the proof of Theorem~\ref{thm:main}:
%\vspace{-0.75pt}
\begin{theorem}\label{thm:tools}
Let $\A$ be as in \eqref{eq:A}, $n>1$. Then $\A$ has FDN if and only if $\A$ is $\C$--elliptic. Moreover, if $\A$ has FDN, then $\A$ is cancelling and there exists a bounded, linear extension operator $E_B\colon\sobo^{\A,1}(\ball)\rightarrow\sobo^{\A,1}(\R^n)$.
\end{theorem}
%\vspace{-0.75pt}
Using the tools from Theorem~\ref{thm:tools}, we can refine our result on fractional scales, thereby obtaining the local versions of the embeddings in \cite[Thm.~8.1, Thm.~8.4]{VS}:
%\vspace{-0.75pt}
\begin{theorem}\label{thm:main_k}
Let $\A$ be as in \eqref{eq:A}, $s\in[k-1,k)$, $q\in(1,\infty)$. Then $\A$ has FDN if and only if there exists $c>0$ such that
\begin{align*}
\|u\|_{\besov^{s,\frac{n}{n-k+s}}_q (\ball,V)}\leq c\left(\|\A u\|_{\lebe^1(\ball,W)}+\|u\|_{\lebe^1(\ball,V)}\right)
\end{align*}
for all $u\in\hold^\infty(\bar{\ball},V)$.
\end{theorem}
%\vspace{-0.75pt}
Here, the Besov spaces on domains are defined as in \cite[Sec.~2]{devoresharpley}.~We obtain the embeddings $\sobo^{\A,1}(\ball)\hookrightarrow\sobo^{s,n/(n-k+s)}(\ball,V)$ if we choose $q=n/(n-k+s)$ (cp. \cite[Thm.~8.1]{VS}). The same method also gives the embedding $\sobo^{\A,1}(\ball)\hookrightarrow\sobo^{k-1,n/(n-1)}(\ball,V)$ (cp. \cite[Thm.~1.3]{VS}). The novelty of Theorems~\ref{thm:main} and \ref{thm:main_k} comes from the fact that, up to our knowledge, there are only a few examples of $\lebe^1$--estimates near the boundary that go in the direction of \cite[Open~Prob.~3]{VS3} (cp. \cite{BrezisVS}). The result of Theorem~\ref{thm:main_k} is sharp on the fractional (or Besov) scale, in the sense that the parameter $s=k$ is ruled out by Ornstein's Non--Inequality. Refinements on Triebel--Lizorkin and Lorentz
space scales follow in the same manner.

This paper is organized as follows: In Section~\ref{sec:prel} we collect preliminaries on function spaces, multi--linear algebra, harmonic analysis, give examples of operators, and collect relevant background results. In Section~\ref{sec:ECvsFDN} we give the proof of the first two statements in Theorem~\ref{thm:tools} and complete the comparison between EC and FDN in Proposition~\ref{prop:intro_EC_vs_FDN}, as well as the comparison between the embeddings \eqref{eq:VS} and \eqref{eq:main_emb}. In Section~\ref{sec:proof} we construct the Jones--type extension and prove Theorems~\ref{thm:main} and \ref{thm:main_k}.

\subsection*{Acknowledgement} The authors wish to thank Jan Kristensen for reading a preliminary version of the paper. The first author gratefully acknowledges financial support from the Hausdorff Center of Mathematics, Bonn. The second author was supported by the Engineering and
Physical Sciences Research Council Award EP/L015811/1. This project has received funding from the European Research Council (ERC) under the European Union's Horizon 2020 research and innovation programme under grant agreement No 757254 (SINGULARITY).
\section{Preliminaries}\label{sec:prel}
Throughout this paper we assume that $n>1$.
\subsection{Function spaces}\label{sec:prelfspaces}
We define, reminiscent of \cite{Mazya}, for $1\leq p\leq\infty$ and open $\Omega\subset\R^n$
\begin{align*}
\sobo^{\A,p}(\Omega)&:=\{u\in\lebe^p(\Omega,V)\colon \A u\in\lebe^p(\Omega,W)\},\\
\bv^{\A}(\Omega)&:=\{u\in\lebe^1(\Omega,V)\colon \A u\in\mathcal{M}(\Omega,W)\},\\
\Vsob^{\A,p}(\Omega)&:=\{u\in\sobo^{\A,p}(\Omega)\colon\nabla^l u\in\lebe^p(\Omega,V\odot^l\R^n),l=1\ldots k-1\},
\end{align*}
and the homogeneous spaces $\dot{\sobo}{^{\A,p}}$ as the closure of $\hold_c^\infty(\R^n,V)$ in the semi--norm $|u|_{\A,p}:=\|\A u\|_{\lebe^p}$. In the case $\A=\nabla^k$, we write $\sobo^{k,p}(\Omega,V)$, $\Vsob^{k,p}(\Omega,V)$. When it is clear from the context what the target space is, we abbreviate the $\lebe^p$--norm of maps defined on $\Omega$ by $\|\cdot\|_{p,\Omega}$. We denote the space of $V$--valued polynomials of degree at most $d$ in $n$ variables by $\R_d[x]^V$. We recall the weighted Bergman spaces $A^p_\alpha(\mathbb{D})$ of holomorphic maps defined on the open unit disc $\mathbb{D}\subset\C$, that are $p$--integrable with weight $w_\alpha(z)=(1-|z|^2)^\alpha$. It is well--known that these are Banach spaces under the $\lebe^p_{w_\alpha}$--norm for $1\leq p<\infty$ and $-1<\alpha<\infty$. We also recall, for $s>0$, $1\leq p,q<\infty$, the Besov space
\begin{align*}
\besov^{s,p}_q(\Omega):=\{u\in\lebe^p(\Omega)\colon |u|_{\besov^{s,p}_q(\Omega)}<\infty\},
\end{align*}
with an obvious choice of norm. Here, the Besov semi--norm is defined (see, e.g., \cite[Sec.~2]{devoresharpley}) for integer $r>s$ by 
\begin{align*}
|u|_{\besov^{s,p}_q(\Omega)}=\|u\|_{\dot{\besov}{_q^{s,p}}(\Omega)}:=\left(\int_0^\infty \dfrac{\sup_{|h|<t}\|\Delta^r_h u\|^q_{\lebe^p(\Omega)}}{t^{1+sq}}\dif t\right)^\frac{1}{q},
\end{align*}
where the $r$-th finite difference $\Delta^r_h u$ is defined to be zero if undefined, i.e., if at least one of $x+jh$, $j=1\ldots r$, falls outside $\Omega$. We also define the homogeneous space $\dot{\besov}{_q^{s,p}}(\R^n)$ as the closure of $\hold^\infty_c(\R^n)$ in the Besov semi--norm.

We also collect the assumptions on our operators. As in Section~\ref{sec:intro}, we say that:
\begin{itemize}
 \item $\A$ is ($\C$--)\emph{elliptic} if and only if the linear map $\A[\xi]\colon V(+\imag V)\rightarrow W(+\imag W)$ is injective for all non--zero $\xi\in\R^n(+\imag\R^n)$;
 \item $\A$ has \emph{FDN} (finite dimensional null--space) if and only if the vector space $\{u\in\mathscr{D}'(\R^n,V)\colon\A u=0\}$ is finite dimensional;
 \item $\A$ is \emph{cancelling} if and only if $\bigcap_{\xi\in S^{n-1}}\A[\xi](V)=\{0\}$.
\end{itemize}
Trivially, $\C$--elliptic operators are elliptic.
\subsection{Multi-linear algebra}
Let $U,V$ be finite dimensional vector spaces and $l\in\mathbb{N}$. We write $\lin(U,V)$ for the space of linear maps $U\rightarrow V$ and $V\odot^l U$ for the space of $V$--valued symmetric $l$--linear maps on $U$%, a subspace of $V\otimes^l U$, the $V$--valued $l$--linear maps on $U$
. This is naturally the space of the $l$--th gradients, i.e., $\D^l f(x)\in V\odot^l U$ for $f\in\hold^l(U,V)$, $x\in U$. For more detail, see \cite[Ch.~1]{Federer}. We also write $a\otimes b=(a_i b_j)$ (the usual tensor product) and $\otimes^l a:=\otimes a\otimes\ldots \otimes a$, where $a$ appears $l$ times on the right hand side. We single out the standard fact that $\widehat{\nabla^l f}(\xi)=\hat{f}(\xi)\otimes^l \xi\in V\odot^l U$ for $f\in\mathscr{S}(U,V)$, $\xi\in U$. We recall the pairing introduced in \cite{BDG}, $v\otimes_\A\xi:=\A[\xi]v$, which is reminiscent of the tensor product notation, i.e., if $\A=\D$, we have $\otimes_\A=\otimes$. We have the following for $k=1$: 
\begin{align*}
\A(\rho u)&=\rho\A u+u\otimes_\A \nabla\rho\qquad\text{for }u\in\hold^1(\R^n,V),\, \rho\in\hold^1(\R^n),\\
\A(\phi(w))&=\phi^\prime(w)\otimes_\A\nabla w\qquad\text{for }\phi\in\hold^1(\R,V),\,w\in\hold^1(\R^n).
\end{align*}
The above can easily be checked by direct computation and will be used without mention.
\subsection{Harmonic analysis}\label{sec:harmonic}
Let $\A$ as in \eqref{eq:A} be elliptic and $u\in\mathscr{S}(\R^n,V)$. We Fourier transform $\A u$ and apply the one--sided inverse $m_\A(\xi):=(\A^*[\xi]\A[\xi])^{-1}\A^*[\xi]\in\lin(W,V)$ of $\A[\xi]$ to get that $\hat{u}(x)=m_\A(\xi)\widehat{\A u}(\xi)$ for $\xi\in\R^n$ (we omitted the complex multiplicative constant arising from Fourier transforming, as it can be absorbed in the definition of $m_\A$). We define the map $\textbf{G}_\A$ as the inverse Fourier transform of the $k$--homogeneous map $m_\A$. Thus we have the Green's function representation $u=\textbf{G}_\A\star\A u$. These considerations are formalized in \cite[Lem.~2.1]{BVS}, an implication of which we recall below:
\begin{lemma}
Let $\A$ as in \eqref{eq:A} be elliptic. Then there exists a $(1-n)$--homogeneous map $\mathbf{K}_\A\in\hold^\infty(\R^n\setminus\{0\},\lin(W,V\odot^{k-1}\R^n))$ such that
\begin{align}\label{eq:representation}
\D^{k-1}u(x)=\int_{\R^n}\mathbf{K}_\A(x-y)\A u(y)\dif y=(\mathbf{K}_\A\star\A u)(x)
\end{align}
for all $u\in\hold^\infty_c(\R^n,V)$.
\end{lemma}
We also record standard facts regarding $\lebe^p$--boundedness of Riesz potentials %and singular integrals 
(see \cite[Ch.~V.1]{Stein} and \cite[Lem.~7.2]{GT}), which are defined by 
\begin{align*}
I_\alpha f:=|\cdot|^{\alpha-n}\star f
\end{align*} for $\alpha\in[0,n)$ and measurable $f\colon\R^n\rightarrow\R$. %If $\alpha=0$, the convolution is understood in the sense of a principal value (singular) integral. 
%\textsc{For further reference, we record from Gilbarg Trudinger \& Adams Hedberg...} 
\begin{theorem}\label{thm:anal_harm}
Let $1\leq p,q\leq \infty$. We have that:
\begin{enumerate}
%\item $I_0$ is bounded on $\lebe^p(\R^n)$ for $1<p<\infty$,
\item $I_\alpha$ is bounded $\lebe^p(\R^n)\rightarrow\lebe^q(\R^n)$ for $0<\alpha<n$, $1<p<n/\alpha$, $q= np/(n-\alpha p)$;
\item\label{itm:riesz_domains} $I_\alpha$ is bounded $\lebe^p(\Omega)\rightarrow\lebe^q(\Omega)$ for $0<\alpha<n$, $0\leq n(1/p- 1/q)<\alpha$ with
\begin{align*}
\|I_\alpha f\|_{\lebe^q(\Omega)}\leq c(\diam\Omega)^{\alpha-n(1/p-1/q)}\|f\|_{\lebe^p(\Omega)}
\end{align*}
for all $f\in\lebe^p(\Omega)$.
\end{enumerate}
\end{theorem}
In Theorem~\ref{thm:anal_harm}\ref{itm:riesz_domains} we make the convention $1/\infty=0$.
\subsection{Examples}\label{sec:examples}
We give examples of operators arising in conductivity, elasticity, plasticity and fluid mechanics (\cite{FM,FS,Milton}). Let $\A$ be as in \eqref{eq:A}. The facts that we use without mention are the main Theorems~\ref{thm:main}, \ref{thm:tools}, and \ref{thm:main_k}.
\begin{enumerate}
\item If $\A=\nabla^k$, we have that $\ker\A=\R_{k-1}[x]^V$, so $\A$ has FDN, hence is EC. This, of course, corresponds to the case of classical Sobolev spaces, but we highlight it here to stress that our generalization brings a new perspective on their study.
\item If $\A u=\mathcal{E}u:=(\nabla u+(\nabla u)^\mathsf{T})/2$ is the symmetrized gradient, it is easy to see that $\ker\A$ is the space of rigid motions, i.e., affine maps of anti--symmetric gradient, so $\A$ has FDN, hence is EC. In this case, we recover the inequality in \cite[Prop.~1.2]{ST}.
\item\label{it:delbar} Let $\A u=\mathcal{E}^D u:=\mathcal{E}u-(\di u/n) \textbf{I}$, where $n\geq2$ and $\textbf{I}$ is the identity $n\times n$ matrix. If $n\geq3$, we have from \cite{Reshet} that $\ker\A$ is the space of conformal Killing vectors, so $\A$ has FDN, hence is EC. If $n=2$, we show in Counterexample~\ref{ex:EC>FDN} that $\A$ is elliptic. However, under the canonical identification $\R^2\cong\C$, we can also identify $\mathcal{E}^D$ with the anti--holomorphic derivative $\bar{\partial}$, so that we can further identify $\ker\A$ with the space of holomorphic functions, so $\A$ does not have FDN. Neither is $\A$ cancelling: by ellipticity, we have that $\mathcal{E}^D[\xi](\R^2)=\R^2$. No critical embedding \eqref{eq:zerotraceemb}, \eqref{eq:ourembedding} can hold in this case.
%\item Recall from \cite[A.2~(2.2)]{FS} that $\sobo^{\di,1}\cap\sobo^{\mathcal{E}^D,1}(\ball)\hookrightarrow\lebe^{n/(n-1)}(\ball,\R^n)$. By \ref{it:delbar}, if $n\geq3$ we can simplify and extend the embedding, whereas if $n=2$ the intersection is necessary. \textcolor{blue}{Guess we can remove that. The space on the left is simply $\bd$.}
\item If $\A=\Delta$, which is clearly elliptic, we have that $\ker\A$ is the space of all harmonic functions, so $\A$ does not have FDN and since $\A[\xi](V)=(\xi_1^2+\ldots+\xi^2_n)\R^N=\R^N$ for $\xi\in\R^n\setminus\{0\}$, neither is $\A$ cancelling.
\item If $\A$ is elliptic, one can consider minimizers of the $\A$--Dirichlet energy $u\mapsto\int_{\ball}|\A u|^2\dif x$, which has Euler--Lagrange system $\A^*\A u=0$. Then $\Delta_\A:=\A^*\A$ is elliptic, as $\langle(\A^*\A)[\xi] v,v\rangle=|\A[\xi]v|^2\gtrsim|\xi|^{2k}|v|^2$, where the last inequality follows from $|\A[\xi]v|>0$ on $\{|\xi|=1,|v|=1\}$ and homogeneity. Therefore $(\A^*\A)[\xi](V)=V$ for all $\xi\neq0$, so the Euler--Lagrange system above has infinite dimensional solution space (by Lemma~\ref{lem:FDNimpliesEC}).
\end{enumerate}
\subsection{Miscellaneous background}
The following relevant facts we quote without proof:
\begin{lemma}[\cite{VS}, Proposition 6.1]\label{lem:canc}
	Let $\A$ as in \eqref{eq:A} be elliptic. Then $\A$ is cancelling if and only if we have that
	\begin{align*}
	\int_{\R^n}\A u\dif x=0
	\end{align*}
	for all $u\in\hold^\infty(\R^n,V)$ such that the support of $\A u$ is compact.
\end{lemma}
\begin{lemma}[Peetre--Tartar Equivalence Lemma, {\cite[Lem.~11.1]{Tartar}}]\label{lem:equivalencelemma}
	Let $E_{1}$ be a Banach space and let $E_{2},E_{3}$ be two normed spaces (with corresponding norms $\|\cdot\|_{i}$, $i\in\{1,2,3\}$) and let $A\in\mathscr{L}(E_{1},E_{2})$ and $B\in\mathscr{L}(E_{1},E_{3})$ be two bounded linear operators such that $B$ is compact and the norms $\|\cdot\|_{1}$ and $\|\cdot\|_{*}:=\|A\cdot\|_{2}+\|B\cdot\|_{3}$ are equivalent on $E_{1}$. Then $\dim(\ker A))<\infty$. 
\end{lemma}
\begin{theorem}[{\cite[Thm.~4]{Kalamajska}}]\label{thm:Ka}
	Let $\A$ as in \eqref{eq:A} be $\C$--elliptic, and $\Omega\subset\R^n$ be a star--shaped domain with respect to a ball. Then there exist an integer $d:=d(\A)$, a linear map $\mathcal{P}\in\lin(\hold^\infty(\bar{\Omega},V),\R_d[x]^V)$ and a smooth map $K\in\hold^\infty(\R^n\times\R^n\setminus\triangle,\lin(W,V))$, where $\triangle=\{(x,x)\colon x\in\R^n\}$ such that $|\D_x^\alpha\D_y^\beta K(x,y)|\lesssim|x-y|^{k-n-|\alpha|-|\beta|}$ for all multi--indices $\alpha,\,\beta$ and all $(x,y)\in\R^n\times\R^n\setminus\triangle$, and 
	\begin{align*}
	u(x)=\mathcal{P}u(x)+\int_{\Omega}K(x,y)\A u(y)\dif y
	\end{align*}
	for all $x\in\Omega$ and $u\in\hold^\infty(\bar{\Omega},V)$. Therefore $\ker\A\subseteq\R_d[x]^V$.
\end{theorem}
\subsection{Other facts about $\sobo^{\A,p}$}
We collect some complementary results that explain, e.g., our choice of definition for the $\A$--Sobolev spaces and of extension technique for $p=1$.
\begin{definition}\label{def:domains}
	A connected open set $\Omega\subset\R^n$ is called a
	\begin{enumerate}
		\item\label{itm:C0} \emph{$\hold^0$--domain} if for any $x\in\partial\Omega$ there exist a neighbourhood $\mathcal{N}$ of $x$ relatively open in $\Omega$, a coordinate system in $\R^n$ and a continuous function $f$ such that, in the new coordinates $(x^\prime,x_n)$, $\mathcal{N}=\{(x',x_n)\colon 0<x_n <f(x^\prime),\,x^\prime\in\ball_1(0)\}$.
		\item \emph{$\hold^{0,1}$--} (or \emph{Lipschitz--})\emph{domain} if $\Omega$ is a $\hold^0$--domain and the function $f$ above can be chosen to be Lipschitz.
		\item \emph{domain with the cone property} if for any $x\in\Omega$ there exists a cone $\mathcal{C}$ with apex at $x$ and a coordinate system with respect to which, for some constants $c_i>0$ we have $\mathcal{C}=\{(x^\prime,x_n)\colon |x^\prime|^2\leq c_1x_n2,\, 0\leq x_n\leq c_2\}$.
		\item \emph{star--shaped domain} (with respect to a ball $\ball\subset\Omega$) if for all $x\in\Omega$, $y\in\ball$, and $0\leq \theta\leq 1$ we have that $\theta x+(1-\theta) y\in\Omega$.
	\end{enumerate}
\end{definition}
We collect a few facts from \cite[Sec.~1.1]{Mazya} on bounded domains, which will be used without mention in the sequel: any star--shaped domain is Lipschitz; Lipschitz domains have the cone property; domains with the cone property can be written as finite unions of star--shaped domains. The following density result closely mimics \cite[Sec.~1.1.4-5]{Mazya}. We reproduce the proof here since, on one hand, it is very elegant and, on the other, it is crucial to prove the extension Theorem~\ref{thm:extension}.
\begin{lemma}\label{lem:density}
	Let $\A$ be as in \eqref{eq:A}, $1\leq p<\infty$, and $\Omega\subset\R^n$ be a bounded $\hold^0$--domain. Then $\hold^\infty(\bar{\Omega},V)$ is dense in $\sobo^{\A,p}(\Omega)$. The same holds true for $\Vsob^{\A,p}(\Omega)$.
\end{lemma}
\begin{proof}
	We only prove the result for $\Vsob^{\A,p}$, the other case following in the same manner.
	
	\emph{Step I}. We first show that $\hold^{\infty}(\Omega,V)\cap\Vsob^{\A,p}(\Omega)$ is dense in $\Vsob^{\A,p}(\Omega)$. This step requires no regularity or boundedness assumption on $\Omega$, other that it is open in $\R^n$.
	
	Consider a Whitney decomposition $\{Q_j\}_{j=1}^\infty$ of $\Omega$ (which is locally finite) \cite{Whitney}, and let $\varepsilon\in(0,1/2)$ and $\rho_j\in\hold^\infty_c(Q_j)$ be a partition of unity associated with this decomposition. We denote by $v_j$ a mollification of $\rho_ju$ such that
	\begin{enumerate}
		\item $\spt v_j$ is also a locally finite cover of $\Omega$;
		\item If $\diam\spt v_j=\lambda_j\diam Q_j$, then %$\lambda_j\bar Q_j$ is a locally finite cover of $\Omega$ and 
		$\lambda_j\downarrow0$;
		\item\label{itm:c} $\|\rho_ju-v_j\|_{\Vsob^{\A,p}(\Omega)}\leq \varepsilon^j$ for all $j\geq 1$. 
	\end{enumerate}
	To make \ref{itm:c} plain, we recall, that mollification and weak derivatives are interchangeble. It follows that $v=\sum v_j\in\hold^\infty(\Omega,V)$ and $u=\sum \rho_ju$ in any compact subset of $\Omega$. Moreover, due to the upper bound on $\varepsilon$, we obtain
	\begin{align*}
	\|u-v\|_{\sobo^{\A,p}(\Omega)}\leq\sum_{j=1}^\infty\|\rho_ju-v_j\|_{\sobo^{\A,p}(\Omega)}\leq \varepsilon(1-\varepsilon)^{-1}\leq 2\varepsilon,
	\end{align*}
	which proves that $v\in\Vsob^{\A,p}(\Omega)$ and concludes the proof of this step.
	
	\emph{Step II}. To conclude the proof of the Lemma, by Step I, it suffices to show density of $\hold^\infty(\bar{\Omega},V)$ in $\hold^\infty(\Omega,V)\cap\Vsob^{\A,p}(\Omega)$.
	
	We cover the boundary of $\Omega$ by open neighbourhoods $\{\mathcal{N}_x\}_{x\in\partial\Omega}$, where each $\mathcal{N}_x$ is the graph of a continuous function as in Definition~\ref{def:domains}\ref{itm:C0}. We extract a finite subcolection $\{\mathcal{N}_j\}_j$ that still covers $\partial\Omega$ and let $\{\rho_j\}_j$ be a partition of unity associated with $\{\mathcal{N}_j\}\cup\mathcal{N}$, where $\Omega\setminus\bigcup\mathcal{N}_j\Subset\mathcal{N}\Subset\Omega$. Since $u=\sum\rho_ju$ in $\Omega$, it suffices to prove the claim for $u$ and $\Omega$ relabelled by $\rho_ju$ and $\mathcal{N}_j\cap \Omega$, respectively. In coordinates $(x^\prime,x_n)$ as in Definition~\ref{def:domains}\ref{itm:C0}, we choose $u_\varepsilon(x^\prime,x_n)=u(x^\prime,x_n-\varepsilon)$ for small enough $\varepsilon>0$. Clearly $u_\varepsilon\in\hold^\infty(\bar{\Omega},V)$ and 
	\begin{align*}
	\|\partial^\alpha u-\partial^\alpha u_\varepsilon\|_{\lebe^p(\Omega,V)}\rightarrow0\qquad\text{as }\varepsilon\downarrow0\text{ whenever }\partial^\alpha u\in\lebe^p(\Omega,V),
	\end{align*}
	which completes the proof.
\end{proof}
\begin{lemma}\label{lem:sob_variants}
	Let $\A$ be as in \eqref{eq:A} have FDN and $\Omega\subset\R^n$ be a bounded Lipschitz domain. Then $\sobo^{\A,p}(\Omega)\simeq\Vsob^{\A,p}(\Omega)$, for each $1\leq p \leq\infty$.
\end{lemma}
\begin{proof}
	One embedding is clear by definition. Conversely, we first prove the inequality under the extra assumption that $\Omega$ is star--shaped with respect to a ball. Let $u\in\sobo^{\A,p}(\ball)$. We use Proposition~\ref{prop:poinc} to estimate, for $1\leq l \leq k-1$,
	\begin{align*}
	\|\nabla^l u\|_{p,\Omega}\leq \|\nabla^l(u-\pi_\Omega u)\|_{p,\Omega}+\|\nabla^l\pi_\Omega u\|_{p,\Omega}\lesssim \|\A u\|_{p,\Omega}+\|\nabla^l\pi_\Omega u\|_{p,\Omega}
	\end{align*}
	To estimate the latter term, we note that $P\mapsto\|\nabla^lP\|_{p,\Omega}$ defines a semi--norm on $\R_d[x]^V$, so it is controlled by the $\lebe^p$--norm of $P$. We have that
	\begin{align*}
	\|\nabla^l\pi_\Omega u\|_{p,\Omega}\lesssim\|\pi_\Omega u\|_{p,\Omega}\leq \|u-\pi_\Omega u\|_{p,\Omega}+\|u\|_{p,\Omega}\lesssim \|\A u\|_{p,\Omega}+\|u\|_{p,\Omega},
	\end{align*} 
	where the last inequality follows by another application of Proposition~\ref{prop:poinc}. Altogether, we have proved that
	\begin{align}\label{eq:sob_variants_star}
	\|u\|_{\Vsob^{\A,p}(\Omega)}\leq C(\Omega)\|u\|_{\sobo^{\A,p}(\Omega)}.
	\end{align}
	
	We now assume just that $\Omega$ is Lipschitz, hence has the cone property. Hence there exist $M$ sub--domains $\Omega_i$ that are star--shaped with respect to a ball and cover $\Omega=\bigcup_{i=1}^M\Omega_i$. We apply \eqref{eq:sob_variants_star} in each $\Omega_i$ to get
	\begin{align*}
	\|u\|_{\Vsob^{\A,p}(\Omega)}\leq\sum_{i=1}^M \|u\|_{\Vsob^{\A,p}(\Omega_i)}\leq \sum_{i=1}^M C(\Omega_i)\|u\|_{\sobo^{\A,p}(\Omega_i)}\leq \max_{1\leq i\leq M}C(\Omega_i)\sum_{i=1}^M\|u\|_{\sobo^{\A,p}(\Omega_i)}.
	\end{align*}
	Let now $1\leq p<\infty$. By concavity of the function $[0,\infty)\ni t\mapsto t^{1/p}$ and Jensen's Inequality, we obtain
	\begin{align}\label{eq:sob_variants_eq2}
	\sum_{i=1}^M\|u\|_{\sobo^{\A,p}(\Omega_i)}\leq M^{1-1/p}\|u\|_{\sobo^{\A,p}(\Omega)}.
	\end{align}
	If $p=\infty$, we simply estimate $\|u\|_{\sobo^{\A,\infty}(\Omega_i)}$ by $\|u\|_{\sobo^{\A,\infty}(\Omega)}$ to note that \eqref{eq:sob_variants_eq2} holds for $p=\infty$ as well (with the convention $1/p=\infty$). The proof is complete.
	%We recall from Theorem \ref{thm:Ka} that $u$ can be represented as $u(x)=\mathcal{P}u(x)+\int_{\ball}K(x,y)\A u(y)\dif y$, where $\mathcal{P}u$ is a polynomial of degree at most $d(\A)$ and $|\D_x^\alpha\D^\beta_y K(x,y)|\lesssim|x-y|^{k-n-|\alpha|-|\beta|}$. Let $1\leq l\leq k-1$. Then
	%\begin{align*}
	%\|\nabla^l u\|_{\lebe^p}\leq \|\nabla^l(u-\mathcal{P}u)\|_{\lebe^p}+\|\nabla^l\mathcal{P}u\|_{\lebe^p}.
	%\end{align*}
	%The first term can easily be controlled by the $\lebe^p$--norm of $\A u$ by Theorem \ref{thm:anal_harm}\ref{itm:riesz_domains} and the growth bounds on the derivatives of $K$ (also see the proof of Proposition~\ref{prop:poinc}). The latter term defines a semi--norm on the space of polynomials of degree at most $d(\A)$, so it can be controlled by the $\lebe^p$--norm. We get
	%\begin{align*}
	%\|\nabla^l u\|_{\lebe^p}\lesssim \|\A u\|_{\lebe^p}+\|\mathcal{P}u\|_{\lebe^p}\lesssim \|\A u\|_{\lebe^p}+\|\mathcal{P}u-u\|_{\lebe^p}+\|u\|_{\lebe^p}\lesssim \|\A u\|_{\lebe^p}+\|u\|_{\lebe^p},
	%\end{align*}
	%which concludes the proof. Here constants depend on the domain.
\end{proof}
\begin{lemma}\label{lem:extp>1}
	Let $\A$ as in \eqref{eq:A} have FDN, $1< p <\infty$, and $\Omega\subset\R^n$ be a star--shaped domain with respect to a ball. Then there exists a bounded, linear extension operator  $E_\Omega\colon\sobo^{\A,p}(\Omega)\rightarrow\Vsob^{k,p}(\R^n,V)$.
\end{lemma}
\begin{proof}
	We use the extension suggested in \cite{Kalamajska94}, namely, in the notation of Theorem~\ref{thm:Ka},
	\begin{align*}
	E_\Omega u(x):=\eta(x)\left(\mathcal{P}u(x)+\int_{\Omega}K(x,y)\A u (y)\dif y\right)
	\end{align*}
	for $u\in\hold^\infty(\bar{\Omega},V)$ and $x\in\R^n$. Here $\eta\in\hold^\infty_c(\R^n)$ is a smooth cut--off that equals 1 in a neighbourhood of $\Omega$. We abbreviate $\mathcal{K}u=\int_{\Omega}K(\,\cdot\,,y)\A u(y)\dif y$. Let $0\leq l\leq k$, and let $\ball$ be a ball containing the support of $\eta$. Then, with domain dependent constants,
	\begin{align*}
	\|\nabla^l E_\Omega u\|_{p,\ball}\lesssim \sum_{j=0}^l \|\nabla^j(\mathcal{P}u+\mathcal{K}u)\|_{p,\ball}\leq \|\mathcal{P}u\|_{\Vsob^{l,p}(\ball,V)}+\sum_{j=0}^l \|\nabla^j\mathcal{K}u\|_{p,\ball}.
	\end{align*}
	We note that $\|\cdot\|_{\Vsob^{l,p}(\ball,V)}$ and $\|\cdot\|_{\lebe^p(\Omega,V)}$ both define norms on $\R_d[x]^V$, hence they are equivalent. We also remark that $\nabla^j\mathcal{K}u=\int_{\Omega}\nabla^j_xK(\,\cdot\,,y)\A u(y)\dif y$, so that 
	\begin{align}\label{eq:est}
	\|\nabla^j\mathcal{K}u\|_{p,\ball}\leq\|\nabla^j\mathcal{K}u\|_{p,\R^n}\lesssim\|\A u\|_{p,\Omega}.
	\end{align}
	If $0\leq j<k$, the proof of \eqref{eq:est} is presented in the proof of Proposition~\ref{prop:poinc}. If $j=k$, \eqref{eq:est} follows from \cite[Ch.~II]{Stein} and the growth bounds on $\nabla^k_x K$.
	%where we use the growth bounds on the derivatives of $K$ and boundedness of Riesz potentials, and, in the case $j=l=k$, of singular integrals (Theorem \ref{thm:anal_harm}). 
	Collecting, we get
	\begin{align*}
	\|\nabla^l E_\Omega u\|_{p,\ball}\lesssim\|\mathcal{P}u\|_{p,\Omega}+\|\A u\|_{p,\Omega}\leq\|\mathcal{P}u+\mathcal{K}u\|_{p,\Omega}+\|\mathcal{K}u\|_{p,\Omega}+\|\A u\|_{p,\Omega}\lesssim\|u\|_{\sobo^{\A,p}(\Omega)},
	\end{align*}
	where the last inequality is obtained from \eqref{eq:est} with $j=0$.
\end{proof}
\begin{lemma}\label{lem:nec_ell}
	Let $\A$ be as in \eqref{eq:A} and $\Omega\subset\R^n$ be a bounded, open set. If $\sobo^{\A,1}(\Omega)\hookrightarrow\sobo^{k-1,p}(\Omega,V)$ for some $p>1$, then $\A$ is elliptic.
\end{lemma}
\begin{proof}
	%The proof is similar to that of Lemma \ref{lem:EC>emb}. 
	Suppose $\A$ is not elliptic, such that there exist $\xi\in \mathbb{S}^{n-1}$, $v\in V\setminus\{0\}$ such that $\A[\xi]v=0$. Consider open cubes $Q_1,\, Q_2$ in $\R^n$ such that $\xi$ is normal to one of their faces and $\bar{Q}_1\subset\Omega\subset Q_2$, of side--lengths $2l_1$, $2l_2$, respectively. We put $u(x)=f((x-x_0)\cdot\xi)v$ for $x_0$ the centre of $Q_1$ and $f(t)=|t|^{k-1-1/p}$ if $t\in\R\setminus\{0\}$. We have that $\A u=0$ %in $\R^n$ 
	and
	\begin{align*}
	\int_{\Omega}|u|\dif x\leq\int_{Q_2}|u|\dif x=\int_{Q_2}|f((x-x_0)\cdot\xi)||v|\dif x=l_2^{n-1}|v|\int_{-l_2}^{l_2}|t|^{k-1-1/p}\dif t<\infty,
	\end{align*}
	so that $u\in\sobo^{\A,1}(\Omega)$. On the other hand,
	\begin{align*}
	\int_{\Omega}|\D^{k-1}u|^p\dif x&\geq\int_{Q_1}|\D^{k-1}u|^p\dif x=\int_{Q_1}|f^{(k-1)}((x-x_0)\cdot\xi)|^p|v\otimes^{k-1}\xi|^p\dif x\\
	&=l_1^{n-1}|v\otimes^{k-1}\xi|^p\int_{-l_1}^{l_1}|t|^{-1}\dif t=\infty,
	\end{align*}
	so that $u\notin\sobo^{k-1,p}_{\locc}(\Omega,V)$. The proof is complete.
	%\begin{align*}
	%\int_{\ball} |u|^q\dif x=\int_{-1}^{1}\int_{\{\xi\}^\perp\cap\ball}|f(y)|^q|v|^q\dif\mathcal{H}^{n-1}\dif y
	%=c(n)|v|^q\int_{-1}^{1}|f(y)|^q(1-y^2)^{\frac{n-1}{2}}\dif y,
	%\end{align*}
	%so that
	%\begin{align*}
	%c(n,q)\int_{-1/2}^{1/2}|f(y)|^q\dif y\leq \|u\|_{\lebe^q(\ball,V)}\leq C(n,q)\int_{-1}^{1}|f(y)|^q\dif y.
	%\end{align*}
	%We now let $f\in\lebe^1(-1,1)$ such that $\int_{-1/2}^{1/2}|f|^p=\infty$, let $\hold^\infty_c(-1,1)\ni\varphi_j\rightarrow f$ in $\lebe^1$, and put $u_j(x)=\varphi_j(x\cdot\xi)v$ for $x\in\ball$. It is then clear that $\A u_j=0$ in the classical sense, and the assumed estimate applied to $u_j$ gives
	%\begin{align*}
	%\int_{-1/2}^{1/2}|\varphi_j(y)|^p\dif y\lesssim\int_{-1}^{1}|\varphi(y)|\dif y,
	%\end{align*}
	%which contradicts the choice of $f$.
\end{proof}
\begin{lemma}\label{lem:embimpliesEC}
	Let $\A$ be as in \eqref{eq:A}. If $\sobo^{\A,1}(\ball)\hookrightarrow\sobo^{k-1,n/(n-1)}(\ball,V)$, then $\A$ is elliptic and cancelling.
\end{lemma}
\begin{proof}
	Necessity of ellipticity follows via Lemma~\ref{lem:nec_ell}. We next show that our assumed embedding implies $\dot{\sobo}{^{\A,1}}(\R^n)\hookrightarrow\dot{\sobo}{^{k-1,n/(n-1)}}(\R^n,V)$ by a scaling argument, so that cancellation follows by the necessity part of \cite[Thm.~1.3]{VS}. Let $u\in\hold^\infty_c(\R^n,V)$ be such that $\spt u\subset\ball_r:=\ball(0,r)$. Then $u_r(x):=u(rx)$ for $x\in\R^n$ is also a test function, with $\spt u_r\subset\ball:=\ball(0,1)$. We estimate, with constants independent of $r$:
	\begin{align*}
	\|\D^{k-1} u\|_{\lebe^{\frac{n}{n-1}}}&=\left(\int_{\ball_r}|\D^{k-1}u(x)|^{\frac{n}{n-1}}\dif x\right)^{\frac{n-1}{n}}=\left(\int_{\ball}r^{\frac{n(k-1)}{n-1}}|\D^{k-1}u_r(y)|^{\frac{n}{n-1}}r^n\dif y\right)^{\frac{n-1}{n}}\\
	&=r^{n-k}\left(\int_{\ball}|\D^{k-1}u_r(y)|^{\frac{n}{n-1}}\dif y\right)^{\frac{n-1}{n}}\leq cr^{n-k}\int_{\ball}|\A u_r(y)|+|u_r(y)|\dif y\\
	&=c\int_{\ball_r}|\A u(x)|\dif x+cr^{-k}\int_{\ball_r}|u(x)|\dif x\leq c\int_{\ball_r}|\A u(x)|\dif x=c\|\A u\|_{\lebe^1},
	\end{align*}
	where the last inequality follows from a change of variable and the Poincar\'e--type inequality with zero boundary values (for elliptic operators)
	\begin{align}\label{eq:zerotracepoinc}
	\|v\|_{\lebe^1(\Omega,V)}\leq c(\diam\Omega)^{k}\|\A v\|_{\lebe^1(\Omega,W)}
	\end{align}
	for all $v\in\hold^\infty_c(\Omega,V)$. The proof is complete.
\end{proof}
The inequality \eqref{eq:zerotracepoinc} follows from an iteration of Poincar\'e's Inequality, the Green--type Formula \eqref{eq:representation}, and Theorem~\ref{thm:anal_harm}\ref{itm:riesz_domains}, in the following way:
\begin{align*}
\|v\|_{\lebe^1(\Omega)}&\leq c(\diam\Omega)^{k-1}\|\D^{k-1}v\|_{\lebe^1(\Omega)}=c(\diam\Omega)^{k-1}\|\mathbb{K}_\A\star(\A v)\|_{\lebe^1(\Omega)}\\
&\leq c(\diam\Omega)^{k-1}\|I_1|\A v|\|_{\lebe^1(\Omega)}\leq c(\diam\Omega)^{k}\|\A v\|_{\lebe^1(\Omega)}
\end{align*}
A similar, straightforward argument also gives the inequality in Lemma~\ref{lem:zerotraceemb} below.

\begin{lemma}\label{lem:zerotraceemb}
	Let $\A$ as in \eqref{eq:A} be elliptic, $k=1$. Then for each $1\leq p<n/(n-1)$, there exists $c>0$ such that 
	\begin{align*}
	\|u\|_{\lebe^p(\ball,V)}\leq c\|\A u\|_{\lebe^1(\ball,W)}
	\end{align*}
	for all $u\in\hold^\infty_c(\ball,V)$.
\end{lemma}
\section{EC Versus FDN}\label{sec:ECvsFDN}

We begin by proving the first two statements in Theorem~\ref{thm:tools}. Throughout, $n>1$.
\begin{proposition}\label{prop:FDNiffTypeC}
Let $\A$ be as in \eqref{eq:A}. Then $\A$ has FDN if and only if $\A$ is $\C$--elliptic.
\end{proposition}
\begin{proof}
From Theorem~\ref{thm:Ka}, we have that if $\A$ is $\C$--elliptic, then $\ker\A$ consists of polynomials of fixed maximal degree. Suppose now that $\A$ is not $\C$--elliptic, so that there exist non--zero $\xi\in\C^n$, $v\in V+\imag V$ such that $\A[\xi]v=0$. We define $u_f(x)=f(x\cdot\xi)v$, for holomorphic $f\colon\C\rightarrow\C$. It can be shown by direct real differentiation of real and imaginary parts and use of the Cauchy--Riemann equations for $f$ that $\D u_f(x)=(\partial_1 f)(x\cdot\xi)v\otimes\xi$. Since $\partial_1 f$ is itself holomorphic, inductively we get that $\D^l u_f(x)=(\partial^l_1f)(x\cdot\xi)v\otimes^l\xi$. We make the simple observation that there exists a linear map $A\in\lin(V\odot^k\R^n,W)$ such that $\A u=A(\D^k u)$, which can be viewed as a coordinate invariant (jet) definition of $\A$. In this notation, by standard properties of the Fourier transform we get $\A[\eta]w=A(w\otimes^k\eta)$ for $\eta\in\R^n$, $w\in V$. It is then easy to see that $\A u_f(x)=(\partial_1^k f)(x\cdot\xi)A(v\otimes^k\xi)=0$. In particular, $\Re u_f,\Im u_f\in\ker\A$, so $\A$ has infinite dimensional null--space.
\end{proof}
The above result enables us to use FDN and $\C$--ellipticity interchangeably. Note that to prove that FDN implies ellipticity, one can simply take real $\xi,\,v$ and $f\in\hold^1(\R)$. We next provide an instrumental ingredient for proving sufficiency of FDN for Theorem~\ref{thm:main_k}.
\begin{lemma}\label{lem:FDNimpliesEC}
Let $\A$ be as in \eqref{eq:A}. If $\A$ has FDN, then $\A$ is cancelling.
\end{lemma}
\begin{proof}
We use Lemma~\ref{lem:canc}. Let $u\in\hold^\infty(\R^n,V)$ be such that $K:=\spt\A u$ is compact. Consider an open ball $\ball$ containing $K$. Cover the complement of  $\ball$ with an increasing chain of overlapping open balls $B_j$ such that $\ball^c\subset \bigcup_j B_j\subset K^c$. In particular, we have $\A u=0$ in each $B_j$, so by Theorem~\ref{thm:Ka}, $u$ must be a polynomial of degree at most $d(\A)$ in each $B_j$. Since the pairs of balls overlap on a set of positive measure, we get that $u$ equals a $V$--valued polynomial $P$ (tacitly viewed as already extended to the entire $\R^{n}$) in $\ball^c$ such that $\A P=0$ in $\R^n$. To conclude, we elaborate on the notation introduced in the proof of Proposition~\ref{prop:FDNiffTypeC}. Put $m:=\dim W$, so that we can write in coordinates $(A\mathscr{V})_{l}=A^l\cdot\mathscr{V}$ for fixed $A^l\in V\odot^k\R^n$, $l=1\ldots m$, and all $\mathscr{V}\in V\odot^k\R^n$. For $l=1\ldots m$, we integrate by parts to get
\begin{align*}
\int_{\R^n}(\A u)_l\dif x&=\int_{\ball}A^l\cdot \D^k u\dif x=
\int_{\partial\ball}A^l\cdot(\D^{k-1}u\otimes\nu) \dif\mathcal{H}^{n-1}\\
&=\int_{\partial\ball}A^l\cdot(\D^{k-1}P\otimes \nu)\dif\mathcal{H}^{n-1}=\int_{\ball}A^l\cdot\D^k P\dif x=\int_{\ball}(\A P)_l\dif x=0,
\end{align*}
where $\nu$ denotes the unit normal to $\partial\ball$. The proof is complete.
\end{proof}
The converse of Lemma~\ref{lem:FDNimpliesEC}, however, is not true in general. In what follows, we complete the algebraic comparison of the FDN condition and \textsc{Van Schaftingen}'s EC condition. We write $N:=\dim V$. The streamline here is that for $N=k=1$, ellipticity alone implies FDN (rendering these cases rather uninteresting), whereas in higher dimensions or for higher orders, there are EC operators that are not FDN. Somewhat surprisingly, there are also a few instances in which ellipticity and $\C$--ellipticity differ, but EC implies FDN. We give the details below.
\begin{lemma}\label{lem:EimpliesFDN}
Let $\A$ as in \eqref{eq:A} be elliptic, $N=k=1$. Then $\A$ has FDN.
\end{lemma}
\begin{proof}
%Say $n=1$. Then $\A[\xi]v=\xi^kA_1v$, so $\A$ is elliptic if and only if $A_1$ is injective on $V$. Let now $\xi+\imag\eta\in\C$ and $v+\imag w\in V+\imag V$ be such that $\A[\xi+\imag\eta](v+\imag w)=0$. Equivalently, $(\xi+i\eta)^k (A_1 v+\imag A_1 w)=0$. We write $a+\imag b:=(\xi+\imag\eta)^k$, so that $aA_1v=bA_2w$ and $bA_1v=-aA_1w$. A simple computation and injectivity of $A_1$ give $a=0=b$ or $v=0=w$, so $\A$ is $\C$--elliptic. 
Since $N=1$, it is clear that $\A$ is $\mathbb{F}$--elliptic, $\mathbb{F}\in\{\R,\C\}$, if and only if the polynomials $(\A [\xi])_l$, $l=1\ldots m$, have no common non--trivial zeroes in $\mathbb{F}$. Since we also assume $k=1$, we have $\A[\xi]=A\xi$ for some $A\in\R^{m\times n}$ %\textcolor{blue}{Something does not fit.}. 
It is clear that all roots of the polynomials thus arising are real (in fact, $\A$ is $\mathbb{F}$--elliptic if and only if $\ker_{\R} A=0$).
\end{proof}
If $n\geq3$, EC turns out to be insufficient for FDN, even for scalar fields or first order operators.
\begin{counterexample}[EC does \emph{not} imply FDN]\label{ex:EC>FDN}
Consider the operators
\begin{align*}
\A_{k,n} u&:=\nabla^{k-1}\left(\partial_1 u_1-\partial_2 u_2, \partial_2 u_1+\partial_1 u_2, \partial_j u_i\right)_{(i,j)\notin\{1,2\}\times\{1,2\}} &\text{ for }N\geq2,\qquad\quad\\
\mathbb{B}_{k,n}u&:=\nabla^{k-2}\left(\partial^2_1 u+ \partial^2_2 u, \partial^2_j u\right)_{j=3\ldots n}&\text{ for }N=1,k\geq2.\,\hspace{1pt}
\end{align*}
If $n\geq3$ or $k\geq2$, then $\A_{k,n}$ is elliptic and cancelling, but has infinite dimensional null--space. The same is true of $\mathbb{B}_{k,n}$ if $n\geq3$ or $k\geq3$.
\end{counterexample}
\begin{proof}
The failure of FDN is clear: simply take 
\begin{align*}
u_\A(x)&:=\left(\Re f\left(x_1+\imag x_2\right), \Im f\left(x_1 + \imag x_2\right), 0,\ldots,0\right)^\mathsf{T}\\
u_\mathbb{B}(x)&:=g(x_1,x_2)
\end{align*}
for holomorphic $f$ and (scalar) harmonic $g$. We next show that $\A_{k,n}=\nabla^{k-1}\A_{1,n}$ is elliptic if $n,N\geq2$. We can reduce to ellipticity of $\A_{1,n}$, since for non--zero $\xi$, we have that $0=\A_{k,n}[\xi]v=(\A_{1,n}[\xi]v)\otimes^{k-1}\xi$, so $\A_{1,n}[\xi]v=0$. Let $1\leq j\leq n$ be such that $\xi_j\neq0$. If $j\geq3$, we clearly get $v=0$. If $1\leq j\leq2$, we get that $v_i=0$ for $3\leq i\leq N$. The remaining equations are $\xi_1v_1-\xi_2v_2=0=\xi_2v_1+\xi_1v_2$, with determinant $\xi_1^2+\xi_2^2>0$, so $v_1=0=v_2$. It remains to check that, under our assumptions, $\A_{k,n}$ is cancelling. The case $k>1$ is easier, since the composition of operators $\mathbb{L}_1\circ\mathbb{L}_2$ is cancelling if $\mathbb{L}_1$ is. This is simply due to the fact that $\textrm{im}(\mathbb{L}_1\circ\mathbb{L}_2)[\xi]=\mathbb{L}_1[\xi](\textrm{im}\mathbb{L}_2[\xi])\subseteq\textrm{im}\mathbb{L}_1[\xi]$. If $k=1$ and $n\geq3$ we can make a straightforward computation. Write $(w_l)_{l=1\ldots Nn-2}:=\A_{1,n}[\xi]v$. For $w\in\bigcap_{\xi\neq0}\A_{1,n}[\xi](V)$, we can essentially test with different values of $\xi\neq0$. By choosing $\xi$ to have exactly one non--zero entry, we obtain that $w_l=0$ for $3\leq l\leq Nn-2$. Incidentally, when testing with $\xi$ such that $\xi_1=0=\xi_2$, we also obtain $w_1=0=w_2$, so all properties are checked for $\A_{k,n}$. Ellipticity of $\mathbb{B}_{k,n}$ is obvious, whereas cancellation is established analogously.
\end{proof}
The two specific cases that are not covered by Lemma~\ref{lem:EimpliesFDN} and Counterexample~\ref{ex:EC>FDN} reveal that the classes EC and FDN can coincide even if they are strictly smaller than the class of elliptic operators.
\begin{lemma}\label{lem:EC=FDN}
Let $n=2$ and $\A$ be as in $\eqref{eq:A}$ be elliptic but not $\C$--elliptic. If any of the following hold,
\begin{enumerate}
\item\label{it:ECimpliesFDN_a} $N=1$, $k=2$,
\item\label{it:ECimpliesFDN_b} $N\geq2$, $k=1$,
\end{enumerate}
then $\A$ is not cancelling.
\end{lemma}
\begin{proof}
Suppose that \ref{it:ECimpliesFDN_a} holds. Since $N=1$ and $\A$ is not $\C$--elliptic, the homogeneous, quadratic, scalar polynomials $(\A[\xi])_l$, $l=1\ldots m$, must have a common complex root. This root cannot be real, as $\A$ is real--elliptic. It follows that $(\A[\xi])_l$ are all multiples of the same quadratic polynomial $P:\R\rightarrow\R$, so that $\A[\xi]v=vP(\xi)w_0$ for all $v\in V\simeq\R$ and some $w_0\in W\setminus\{0\}$. It is clear then that $\A[\xi](V)=\R w$ for all $\xi\neq0$. We next assume that \ref{it:ECimpliesFDN_b} holds. Since $\A$ is elliptic, there exist linearly independent $\xi,\eta\in\R^2$, $v,w\in\R^N$ such that $\A[\xi]v=\A[\eta]w$ and $\A[\xi]w=-\A[\eta]v$. We also have that any $\zeta\in\R^2$ can be written as $\zeta=a\xi+b\eta$. We put $v_\zeta:=a v+b w$. It follows that 
\begin{align*}
\A[\zeta]v_\zeta=\A[a\xi+b\eta](a v+b w)=(a^2+b^2)\A[\xi]v,
\end{align*}
so that $\bigcap_{\zeta\in\R^2\setminus\{0\}}\A[\zeta](V)\ni\A[\xi]v\neq0$.
\end{proof}
We conclude this section with a minor curiosity: we can append the proof above by taking $w_\zeta:=b v-a w$ and obtain $\bigcap_{\zeta\neq0}\A[\zeta](V)\supset\{\A[\xi]v,\A[\xi]w\}$. Therefore, if $n=2$, $k=1$, and $\A$ is elliptic but not cancelling, then 
\begin{align*}
\dim\bigcap_{\zeta\neq0}\A[\zeta](V)\geq2.
\end{align*}
This is no longer the case in higher dimensions, as can be seen by considering the first order operator $\A u=(\di u,\,\curl u)$ for $u\colon\R^n\rightarrow\R^n$, $n\geq3$.
%\begin{align*}
%\A u =\left(\begin{matrix}
%\di u\\
%\curl u
%\end{matrix}
%\right)\text{ for }u\colon\R^n\rightarrow\R^n, \,n\geq3.
%\end{align*}
\subsection{Insufficiency of EC}\label{sec:EC>emb}
We next give examples of first order EC operators and domains $\Omega\subset\R^n$ for which the Sobolev--type embedding fails. Firstly, we pause to compare the embeddings in Theorem~\ref{thm:main}\ref{it:main_b}, \ref{it:main_c} with \textsc{Van Schaftingen}'s homogeneous embedding $\dot{\sobo}{^{\A,1}}(\R^n)\hookrightarrow\lebe^{\frac{n}{n-1}}(\R^n,V)$. For elliptic $\A$, the latter embedding is equivalent to 
\begin{align}\label{eq:zerotraceemb}
\sobo^{\A,1}_0(\ball)\hookrightarrow\lebe^\frac{n}{n-1}(\ball,V)
\end{align}
(see Lemma~\ref{lem:embimpliesEC} for a scaling argument) and it can easily be shown that, in the absence of cancellation, we can still prove by means of a Green's formula and boundedness of Riesz potentials that $\sobo^{\A,1}_0(\ball)\hookrightarrow\lebe^p(\ball,V)$ for any $1\leq p < n/(n-1)$ (see Lemma~\ref{lem:zerotraceemb}). Here $\sobo^{\A,p}_0(\ball)$ is defined as the closure of $\hold_c^\infty(\ball,V)$ in the (semi--)norm $u\mapsto\|\A u\|_{\lebe^p}$. The situation is dramatically different as far as $\lebe^p$--embeddings of $\sobo^{\A,1}(\ball)$ are concerned. By Theorem~\ref{thm:main}, if the critical embedding 
\begin{align}\label{eq:ourembedding}
\sobo^{\A,1}(\ball)\hookrightarrow\lebe^\frac{n}{n-1}(\ball,V)
\end{align}
fails, then no uniform higher integrability estimate is possible. The difference can be even sharper: for EC, non--FDN operators there are maps in $\sobo^{\A,1}(\ball)$ that have no higher integrability, so the homogeneous embedding \eqref{eq:zerotraceemb} can hold even if the inhomogeneous \eqref{eq:ourembedding} fails completely. We highlight that the main difference between $\sobo^{\A,1}(\ball)$ and $\sobo^{\A,1}_0(\ball)$ lies in the traces, which are integrable if and only if $\A$ has FDN \cite{BDG}. 

The existence of elliptic and cancelling $\A$, domains $\Omega$, and of maps $u\in\sobo^{\A,1}(\Omega)$ that are in no $\lebe^p(\Omega,V)$, $p>1$, follows from Counterexample~\ref{ex:EC>FDN} above and the next Lemma, which is a strengthened version of the strict inclusion of (weighted) Bergman spaces generalized to elliptic, non--FDN operators.
\begin{lemma}\label{lem:EC>emb}
Let $k=1$ and $\A$ as in \eqref{eq:A} be elliptic but \emph{not} have FDN, so there exist linearly independent $\eta_1,\eta_2\in\R^n$ such that $\A[\eta_1+\imag\eta_2]$ has non--trivial kernel in $V+\imag V$. Assume that $\eta_1,\eta_2$ are orthonormal. If any of the following holds:
\begin{enumerate}
\item\label{itm:cylinder} $\Omega:=\ball_{\spano\{\eta_1,\eta_2\}}\times[0,1]^{n-2}$,
\item\label{itm:ball} $\Omega:=\ball$,
\end{enumerate}
then there exists smooth $u\in\lebe^1\setminus\bigcup_{p>1}\lebe^p(\Omega,V)$ such that $\A u=0$.
\end{lemma}
\begin{proof}
We write  $\xi=\eta_1+\imag\eta_2$, and write $D$ for the unit disc in $\spano\{\eta_1,\eta_2\}$. We stress that each $\eta_j$ must be non--zero by ellipticity of $\A$, so $D$ is indeed a non--degenerate disc. We also know from the proof of Proposition~\ref{prop:FDNiffTypeC} that there exist non--zero $v\in V+\imag V$ such that $\A[\xi]v=0$, and one can show by direct computation that for any holomorphic function $f$ we can define $u_f(x):=f(x\cdot\xi)v$, for which $\A\Re u_f=0=\A\Im u_f$. We have that 
\begin{align*}
\int_{\Omega}|u_f(x)|^p\dif x&=\int_D\int_{(\eta+\{\eta_1,\eta_2\}^\perp)\cap\Omega}|f(\eta\cdot\xi)|^p|v|^p\dif\mathcal{H}^{n-2}\dif\mathcal{H}^2(\eta)\\
&=|v|^p\int_D |f(\eta\cdot\xi)|^p\mathcal{H}^{n-2}\left((\eta+\{\eta_1,\eta_2\}^\perp)\cap\Omega\right)\dif\mathcal{H}^2(\eta)
\end{align*}
We now make the case distinction. Assume \ref{itm:cylinder} holds, so
\begin{align*}
\int_{\Omega}|u_f(x)|^p\dif x&=|v|^p\int_D |f(\eta\cdot\xi)|^p\dif\mathcal{H}^2(\eta)=\int_{\mathbb{D}}|f(z)|^p\dif\mathcal{L}^2(z).
\end{align*}
Assume \ref{itm:ball}, so
\begin{align*}
\int_{\Omega}|u_f(x)|^p\dif x&=c(n)|v|^p\int_D |f(\eta\cdot\xi)|^p(1-|\eta|^2)^\frac{n-2}{2}\dif\mathcal{H}^2(\eta)\\
&=c(n)|v|^p\int_{\mathbb{D}}|f(z)|^p(1-|z|^2)^\frac{n-2}{2}\dif\mathcal{L}^2(z),
\end{align*}
where $c(n)$ denotes the volume of the $(n-2)$--dimensional ball. By Lemma~\ref{lem:baire} below, we can choose $f\in A^1_{\alpha}(\mathbb{D})\setminus\bigcup_{p>1}A^p_\alpha(\mathbb{D})$ for $\alpha=0$ and $\alpha=(n-2)/2$ respectively, so that both $\Re u_f$ and $\Im u_f$ are in $\lebe^1(\ball,V)$, but one of them is in not in any other $\lebe^p$. This proves the claim.
\end{proof}
The following Lemma is also feasible by direct computation, but we prefer to give an abstract argument for the sake of brevity.
\begin{lemma}\label{lem:baire}
For all $1\leq p<\infty$, $\alpha\geq0$ the set $A^p_\alpha(\mathbb{D})\setminus\bigcup_{q>p}A^q_\alpha(\mathbb{D})$ is non--empty.
\end{lemma}
\begin{proof}
We abbreviate $A^p:=A_\alpha^p(\mathbb{D})$. The proof relies on the strict inclusion $A^q\subsetneq A^p$ for $1\leq p<q<\infty$ proved in \cite[Cor.~68]{ZZ} and a Baire category argument. Assume that the result is false, so that by H\"older's Inequality we can find a sequence $q_j\downarrow p$ such that $A^p=\bigcup_j A^{q_j}$. For natural $l$, we define the sets $F_l^j:=\{f\in A^{q_j}\colon\|f\|_{A^{q_j}}^{q_j}\leq l\}$, which we claim to be closed in $A^p$. Let $f_m\in F^j_l$ converge to $f$ in $A^p$. By completeness of $A^p$, we have, by Fatou's Lemma on a pointwise convergent, not relabelled subsequence that 
\begin{align*}
\int_{\mathbb{D}}|f|^{q_j}w_\alpha\dif\mathcal{L}^2\leq\liminf_{m\rightarrow\infty}\int_{\mathbb{D}}|f_m|^{q_j}w_\alpha\dif\mathcal{L}^2\leq l,
\end{align*}
so that indeed $f\in F_l^j$. Since $A^{q_j}$ is a proper subspace of $A^p$, it follows that the sets $F^j_l$ are nowhere dense in $A^p$. It remains to notice that then $A^p=\bigcup_{j,l}F^j_l$, which contradicts completeness of $A^p$ by Baire's Theorem.
\end{proof}
\subsection{Comparison to the {Bourgain}--{Brezis} condition}
We recall here the assumptions on $\A$ (sufficient for EC) under which a general inequality of the type \eqref{eq:VS} was first proved in \cite{BB07}, in the case $k=1$ and $V=\R^n$. In their notation, we write $(\A u)_s=\langle L^{(s)},\nabla u\rangle$ for matrices $L^{(s)}\in\R^{n\times n}$, $s=1\ldots m$. It is shown in \cite[Thm.~25]{BB07}, that if an operator $\A$ is elliptic such that $\det L^{(s)}=0$ for $s=1\ldots m$, then \eqref{eq:VS} holds. It is clear (either by \cite[Thm.~1.3]{VS} or by direct computation) that such operators are cancelling. By Lemma~\ref{lem:EC=FDN}, if $n=2$, we have that such $\A$ also has FDN, and thus satisfies \eqref{eq:ourembedding}. However, if $n\geq3$, we show that $\A_{1,n}$ as in Counterexample~\ref{ex:EC>FDN} with $N=n$ satisfies the \textsc{Bourgain}--\textsc{Brezis} condition, but do not have FDN. We explicitly write down the matrices $L^{(s)}$ if $n=3$, the general case being a simple exercise:
\begin{align*}
&\left(\begin{array}{ccc}1 & 0 & 0 \\ 0 & -1 & 0\\ 0 & 0&0\end{array}\right),
\left(\begin{array}{ccc}0 & 1 & 0 \\ 1 & 0 & 0\\ 0 & 0&0\end{array}\right),
\left(\begin{array}{ccc}0 & 0 & 0 \\ 0 & 0 & 0\\ 1 & 0&0\end{array}\right),
\left(\begin{array}{ccc}0 & 0 & 0 \\ 0 & 0 & 0\\ 0 & 1&0\end{array}\right),\\
&\left(\begin{array}{ccc}0 & 0 & 0 \\ 0 & 0 & 0\\ 0 & 0&1\end{array}\right),
\left(\begin{array}{ccc}0 & 0 & 1 \\ 0 & 0 & 0\\ 0 & 0&0\end{array}\right),
\left(\begin{array}{ccc}0 & 0 & 0 \\ 0 & 0 & 1\\ 0 & 0&0\end{array}\right).
\end{align*}
By the reasoning in Section~\ref{sec:EC>emb}, with $\A=\A_{1,n}$, we have that $\dot{\sobo}{^{\A,1}}(\R^n)\hookrightarrow\lebe^{n/(n-1)}(\R^n)$, but there are maps in $\sobo^{\A,1}(\ball)$ that have no higher integrability.
\section{The Sobolev--type Embedding on Domains}\label{sec:proof}
\subsection{A Jones--type Extension}
In this section we complete the proof of Theorem~\ref{thm:tools} with the following generalization:
\begin{theorem}\label{thm:extension}
Let $\A$ as in \eqref{eq:A} have FDN, $1\leq p <\infty$, $\Omega\subset\R^n$ be a %n $(\varepsilon,\delta)$--
bounded Lipschitz domain. Then there exists a bounded, linear extension operator  
\begin{align*}
E_\Omega\colon\sobo^{\A,p}(\Omega)\rightarrow\Vsob^{\A,p}(\R^n).
\end{align*}
\end{theorem}
To prove this result we use \textsc{Jones}' method of extension developed in the celebrated paper \cite{Jones}. We stress that the need to use this technically involved method arose due to unboundedness of singular integrals on $\lebe^1$ (cp. Lemma~\ref{lem:extp>1} for $p>1$) and not out of our desire to deal with very rough domains, which are not explicitly covered here. Recall that \textsc{Jones}'s original idea was to decompose a small neighbourhood of $\partial\Omega$ into small cubes and assign suitable polynomials of degree at most $k-1$ to each cube. Inspired by \cite[Sec.~4.1-2]{BDG}, we assign elements of $\ker\A$ on such cubes, as explained below. We stress that the fact that $\ker \A$ consists of a finite dimensional space of polynomials is essential for the construction to work. With this crucial modification, the streamlined proof that we include below mostly follows the same lines as in \cite[Sec.~2-3]{Jones}, where all the details we omit can be found.  What deserves some special attention is a Poincar\'e--type inequality, which is interesting in its own right, as it implies that $\sobo^{\A,p}(\ball)\simeq\Vsob^{\A,p}(\ball)$ for FDN operators (see Lemma~\ref{lem:sob_variants}). We present it below and mention that it is a generalization of the results in \cite[Sec.~1.1.11]{Mazya} and \cite[Sec.~3]{BDG}. We extend the notation presented in Theorem~\ref{thm:Ka} by $\pi_\Omega u:=\Pi \mathcal{P} u$, where $\Pi$ denotes the $\lebe^2$--orthogonal projection of $\R_d[x]^V$ onto $\ker\A$.
\begin{proposition}[Poincar\'e--type inequality]\label{prop:poinc}
Let $\A$ as in \eqref{eq:A} have FDN, $1\leq p\leq\infty$, $0\leq l<k$, and $\Omega\subset\R^n$ be a star--shaped domain with respect to a ball. Then there exists $c>0$ such that
\begin{align}\label{eq:poinc}
\|\nabla^l(u-\pi_\Omega u)\|_{p,\Omega}\leq c(\diam\Omega)^{k-l}\|\A u\|_{p.\Omega}
\end{align}
for all $u\in\hold^\infty(\bar{\Omega},V)$.
\end{proposition}
Interestingly, $\A$ having FDN is not necessary for the estimate \eqref{eq:poinc} to hold, as can be seen from \cite{Fuchs1}. We believe that ellipticity alone is sufficient for the estimate to hold and intend to pursue this in future work.
\begin{proof}
We start with $\|\nabla^l(u-\pi_\Omega u)\|_{p,\Omega}\leq\|\nabla^l(u-\mathcal{P} u)\|_{p,\Omega}+\|\nabla^l(\mathcal{P}u-\pi_\Omega u)\|_{p,\Omega}$, and estimate both terms. 
We have by the growth conditions on $K$ from Theorem~\ref{thm:Ka} that
\begin{align*}
\|\nabla^l(u-\mathcal{P} u)\|_{p,\Omega}&=\left(\int_{\Omega}\left\vert\int_{\Omega}\nabla^l_x K(x,y)\A u(y)\dif y\right\vert^p\dif x\right)^{\frac{1}{p}}\\
&\lesssim\left(\int_{\Omega}\left(\int_{\Omega}\dfrac{|\A u(y)|}{|x-y|^{n+l-k}}\dif y\right)^p\dif x\right)^{\frac{1}{p}}.
\end{align*}
Now consider the case $n+l>k$, so that we can estimate the RHS using Theorem~\ref{thm:anal_harm}\ref{itm:riesz_domains}:
\begin{align*}
\|I_{k-l}(|\A u|)\|_{p,\Omega}\lesssim(\diam\Omega)^{k-l}\|\A u\|_{p,\Omega}.
\end{align*}
If, on the other hand, $n+l\leq k$, with $R=\diam\Omega$, we perform the elementary estimation
\begin{align*}
\left(\int_{\Omega}\left(\int_{\Omega}|\A u(y)||x-y|^{k-l-n}\dif y\right)^p \dif x\right)^{\frac{1}{p}}&\leq R^{k-l-n}\left(\int_\Omega|\A u(y)|\dif y\right)\left(\int_\Omega\dif x\right)^\frac{1}{p}\\
&\leq R^{k-l-n}\cdot R^{n(p-1)/p}\cdot\|\A u\|_{p,\Omega}\cdot R^{n/p}\\
&=R^{k-l}\|\A u\|_{p,\Omega},
\end{align*}
%and we obtain the estimate by standard boundedness of Riesz potentials (see  for the precise scaling if $n+l-k>0$; the case $n+l-k\leq0$ follows by H\"older's Inequality). 
where obvious modifications have to be made if $p=\infty$.

We then note that $P\mapsto\|P-\Pi P\|_{p,\Omega}$ and $P\mapsto\|\A P\|_{p,\Omega}$ respectively define a semi--norm and a norm on the finite dimensional vector space $\R_d[x]^V/\ker\A$, so that the second term $\|\nabla^l(\mathcal{P}u-\pi_\Omega u)\|_{p,\Omega}\lesssim\|\A\mathcal{P}u\|_{p,\Omega}$, with a domain dependent constant. We recall from the original proof of Theorem~\ref{thm:Ka} that $\mathcal{P}u$ is the averaged Taylor polynomial
\begin{align*}
\mathcal{P}u(x)=\int_{\Omega}\sum_{|\alpha|\leq d} \frac{\partial^\alpha_y\left((y-x)^\alpha w(y)\right)}{\alpha!}u(y)\dif y=\int_{\Omega}\sum_{|\alpha|\leq d}\frac{(x-y)^{\alpha}}{\alpha!} w(y)\partial^\alpha u(y)\dif y,
\end{align*}
where the weight $w$ is a smooth map supported in the ball with respect to which $\Omega$ is star--shaped such that $\int w=1$. One can show by direct computation that averaged Taylor polynomials ``commute'' with derivatives, in the sense that
\begin{align*}
\A\mathcal{P}u=\int_{\Omega} \sum_{|\beta|\leq d-k} \frac{\partial^\beta_y\left((y-\cdot\,)^\beta w(y)\right)}{\alpha!}\A u(y)\dif y.
\end{align*}
It is then obvious that $\|\A\mathcal{P}u\|_{p,\Omega}\lesssim\|\A u\|_{p,\Omega}$. The precise dependence of the constant on the domain follows by standard scaling arguments.
\end{proof}
We next introduce the framework required to prove Theorem~\ref{thm:extension}. We use the same Whitney coverings as in \cite{Jones}, which we recall for the reader's convenience. Firstly recall the Decomposition Lemma introduced in \cite{Whitney}, that any open subset $\Omega\subset\R^n$ can be covered with a countable collection $\mathcal{W}_1:=\{S_j\}$ of closed dyadic cubes satisfying
\begin{enumerate}
\item[$(\mathrm{D}_1)$] $\ell(S_j)/4\leq \ell(S_l)\leq 4 \ell(S_j)$ if $S_j\cap S_l\neq\emptyset$,
\item[$(\mathrm{D}_2)$] $\inte S_j\cap\inte S_l=\emptyset$ if $j\neq l$,
\item[$(\mathrm{D}_3)$] $\ell(S_j)\leq\mathrm{dist}(S_j,\partial\Omega)\leq 4\sqrt{n}\ell(S_j)$ for all $j$,
\end{enumerate}
where $\ell(Q)$ denotes the side--length of a cube $Q$. We henceforth assume that $\Omega$ is as in the statement of Theorem~\ref{thm:extension}%, so in particular $\Omega$ is an $(\varepsilon,\delta)$--domain
. We further consider a Whitney decomposition $\mathcal{W}_2:=\{Q_l\}$ of $\R^n\setminus\bar{\Omega}$, and further define $\mathcal{W}_3:=\{Q\in\mathcal{W}_2\colon \ell(Q)\leq\varepsilon\delta/(16n)\}$. We reflect each cube $Q\in\mathcal{W}_3$ to a non--unique interior cube $Q^*\in\mathcal{W}_1$ such that 
\begin{enumerate}
\item[$(\mathrm{R}_1)$] $\ell(Q)\leq\ell(Q^*)\leq4\ell(Q)$,
\item[$(\mathrm{R}_2)$] $\mathrm{dist}(Q,Q^*)\leq C\ell(Q)$,
\end{enumerate}
where above and in the following $C$ denotes a constant depending on $k,p,n,\varepsilon,\delta$ only; additional dependencies will be specified. The non--uniqueness causes no issues, as one can show that
\begin{enumerate}
\item[$(\mathrm{R}_3)$] For any two choices $S_1,S_2$ of $Q^*$, we have $\mathrm{dist}(S_1,S_2)\leq C\ell(Q)$,
\item[$(\mathrm{R}_4)$] For any $S\in\mathcal{W}_1$, there are at most $C$ cubes $Q\in\mathcal{W}_3$ such that $S=Q^*$,
\item[$(\mathrm{R}_5)$] For any adjacent $Q_1,Q_2\in\mathcal{W}_3$, we have $\mathrm{dist}(Q_1^*,Q_2^*)\leq C\ell(Q_1)$.
\end{enumerate}
For detail on theses basic properties of the reflection see \cite[Lem.~2.4-7]{Jones}. We conclude the presentation of the decomposition by quoting the following:
\begin{lemma}[{\cite[Lem.~2.8]{Jones}}]
For any adjacent cubes $Q_1,Q_2\in\mathcal{W}_3$, there is a chain $\mathcal{C}(Q^*_1,Q^*_2):=\{Q_1^*=:S_1,S_2,\ldots S_m:=Q^*_2\}$ of cubes in $S_j\in\mathcal{W}_1$, i.e., such that $S_j$ and $S_{j+1}$ touch %on an $(n-1)$--dimensional hyper--surface 
for all $j$, and $m\leq C$.
\end{lemma}
We proceed to define the extension operator 
\begin{align*}
E_\Omega u:=
\begin{cases}
u&\text{in }\Omega\\
\sum_{Q\in\mathcal{W}_3}\varphi_Q \pi_{Q^*}u&\text{in }\R^n\setminus\bar{\Omega},
\end{cases}
\end{align*}
where $\{\varphi_Q\}_{Q\in\mathcal{W}_3}\subset\hold^\infty(\R^n)$ is a partition of unity such that for all $Q\in\mathcal{W}_3$ we have
\begin{enumerate}
\item[$(\mathrm{P}_1)$] $0\leq\varphi_Q\leq1$ and $\sum_{Q\in\mathcal{W}_3}\varphi_Q=1$ in $\bigcup\mathcal{W}_3$,
\item[$(\mathrm{P}_2)$] $\spt\varphi_Q\subset17/16Q$, where $\lambda Q$ denotes the homothety of $Q$ by $\lambda$ about its centre,
\item[$(\mathrm{P}_3)$] $|\nabla^l\varphi_Q|\leq C\ell(Q)^{-l}$ for all $0\leq l\leq k$.
\end{enumerate}
Our proof mostly follows the lines of the original proof. We first prove an estimate on chains in $\mathcal{W}_1$, then suitably bound the norms of the derivatives in the exterior domains, and we conclude by showing that the extension has weak derivatives in full--space. We warn the reader that in the remainder of this section we may use the properties of the decomposition, reflection and partition of unity without mention.
\begin{lemma}[{\cite[Lem.~3.1]{Jones}}]\label{lem:chain}
Let $\mathcal{C}:=\{S_1,\ldots S_m\}\subset\mathcal{W}_1$ be a chain. Then for $0\leq l<k$ we have
\begin{align*}
\|\nabla^l(\pi_{S_1}u-\pi_{S_m}u)\|_{p,S_1}\leq C(m)\ell(S_1)^{k-l}\|\A u\|_{p,\cup\mathcal{C}}
\end{align*}
for all $u\in\hold^\infty(\bar{\Omega},V)$.
\end{lemma}
\begin{proof}
We remark that $\lebe^p$--norms of polynomials of degree at most $d$ on adjacent cubes in $\mathcal{W}_1$ are comparable (see, e.g., \cite[Lem.~2.1]{Jones}). We get
\begin{align*}
\mathrm{LHS}&\leq \sum_{j=1}^{m-1}\|\nabla^l(\pi_{S_{j+1}}u-\pi_{S_j}u)\|_{p,S_1}\\
&\leq C(m) \sum_{j=1}^{m-1}\|\nabla^l(\pi_{S_{j+1}}u-\pi_{S_j\cup S_{j+1}}u)\|_{p,S_{j+1}}+\|\nabla^l(\pi_{S_j\cup S_{j+1}}u-\pi_{S_j}u)\|_{p,S_j}\\
&\leq C(m)\sum_{j=1}^{m-1}\left(\|\nabla^l(\pi_{S_{j+1}}u-u)\|_{p,S_{j+1}}+2\|\nabla^l(u-\pi_{S_j\cup S_{j+1}}u)\|_{p,S_j\cup S_{j+1}}\right.\\
&\left.+\|\nabla^l(u-\pi_{S_j}u)\|_{p,S_j}\right)
\end{align*}
and we can use the Poincar\'e--type inequality, Proposition~\ref{prop:poinc}, to conclude.
\end{proof}
\begin{lemma}[{\cite[Prop.~3.4]{Jones}}]\label{lem:localbounds}
For $1\leq p\leq\infty$, we have $\|E_\Omega u\|_{\Vsob^{\A,p}(\R^n\setminus\bar{\Omega})}\leq C \|u\|_{\sobo^{\A,p}(\Omega)}$ for all $u\in\hold^\infty(\bar{\Omega},V)$.
\end{lemma}
\begin{proof}
We estimate on each cube in $\mathcal{W}_2$, distinguishing between small and large cubes. We also distinguish between $\A$ and the derivatives of order less than $k$. Let $Q_0\in\mathcal{W}_3$. Then, since $\varphi_Q$ sum to one in $Q_0$ and $\A\pi_{Q_0^*}u\equiv0$, we have
\begin{align*}
\|\A E_\Omega u\|_{p,Q_0}&=\left\|\A \sum_{Q\in\mathcal{W}_3}\varphi_Q( \pi_{Q^*}u-\pi_{Q_0^*}u)\right\|_{p,Q_0}\\
&\leq\left\|\sum_{\emptyset\neq Q_0\cap Q\in\mathcal{W}_3}\A(\varphi_Q( \pi_{Q^*}u-\pi_{Q_0^*}u))\right\|_{p,Q_0}\\
&\leq C\sum_{\emptyset\neq Q_0\cap Q\in\mathcal{W}_3}\sum_{j=0}^{k-1}\||\nabla^{k-j}\varphi_Q| |\nabla^j(\pi_{Q^*}u-\pi_{Q_0^*}u)|\|_{p,Q_0}\\
&\leq C\sum_{\emptyset\neq Q_0\cap Q\in\mathcal{W}_3}\sum_{j=0}^{k-1}\ell(Q_0)^{j-k}\|\nabla^j(\pi_{Q^*}u-\pi_{Q_0^*}u)\|_{p,Q_0^*}\\
&\leq C\sum_{\emptyset\neq Q_0\cap Q\in\mathcal{W}_3}\|\A u\|_{p,\cup\mathcal{C}(Q_0^*,Q^*)},
\end{align*}
where the last inequality follows from Lemma~\ref{lem:chain}. With a similar reasoning we obtain, for $0\leq l\leq k-1$, that 
\begin{align*}
\|\nabla^l E_\Omega u\|_{p,Q_0}\leq C\left(\|\nabla^l u\|_{p,Q_0^*}+\ell(Q_0)^{k-l}\sum_{\emptyset\neq Q_0\cap Q\in\mathcal{W}_3}\|\A u\|_{p,\cup\mathcal{C}(Q_0^*,Q^*)}\right).
\end{align*}
 We move on to the case $Q_0\in\mathcal{W}_2\setminus\mathcal{W}_3$, so if $Q\cap Q_0\neq\emptyset$, then $\ell(Q)\geq\ell(Q_0)/4\geq\varepsilon\delta/(64n)\geq C$. Let $0\leq l\leq k-1$, so that
\begin{align*}
\|\nabla^l E_\Omega u\|_{p,Q_0}&\leq\sum_{\emptyset\neq Q_0\cap Q\in\mathcal{W}_3}\|\nabla^l(\varphi_Q \pi_{Q^*}u)\|_{p,Q_0}\\
&\leq C\sum_{\emptyset\neq Q_0\cap Q\in\mathcal{W}_3}\sum_{j=1}^l \ell(Q_0)^{j-l}\|\nabla^{j} \pi_{Q^*}u\|_{p,Q_0}\\
&\leq C\sum_{\emptyset\neq Q_0\cap Q\in\mathcal{W}_3}\sum_{j=1}^l \ell(Q_0)^{j-l}\|\nabla^{j} \pi_{Q^*}u\|_{p,Q^*}\\
&\leq C\sum_{\emptyset\neq Q_0\cap Q\in\mathcal{W}_3}\sum_{j=1}^l \ell(Q_0)^{j-l}(\|\nabla^{j}( \pi_{Q^*}u-u)\|_{p,Q^*}+\|\nabla^{j}u\|_{p,Q^*})\\
&\leq C\sum_{\emptyset\neq Q_0\cap Q\in\mathcal{W}_3}\|u\|_{\Vsob^{l,p}(Q^*,V)}+\ell(Q_0)^{k-l} \|\A u\|_{p,Q^*}.
\end{align*}
As, above, we similarly show that $\|\A E_\Omega u\|_{p,Q_0}\leq C\sum_{\emptyset\neq Q_0\cap Q\in\mathcal{W}_3} \|u\|_{\Vsob^{\A,p}(Q^*)}$. There is no loss in assuming that $\ell(Q_0)\leq1$ for any $Q_0\in\mathcal{W}_2$, so that we can collect the estimates to obtain
\begin{align*}
\|E_\Omega u\|_{\Vsob^{\A,p}(Q_0)}\leq C \sum_{\emptyset\neq Q_0\cap Q\in\mathcal{W}_3} \|u\|_{\Vsob^{\A,p}(\mathcal{C}(Q_0^*,Q^*))}.
\end{align*}
It remains to use local finiteness of the partition of unity (see, e.g., \cite[Eqn.~(3.1-4)]{Jones}) and Lemma~\ref{lem:sob_variants} to conclude.
\end{proof}
\begin{proof}[Proof of Theorem~\ref{thm:extension}]
We firstly show that $E_\Omega u$ has weak derivatives in $\R^n$, for which it suffices (by Lemma~\ref{lem:density}) to show that $E_\Omega$ maps $u\in\Vsob^{k,\infty}(\bar{\Omega},V)$ into $\Vsob^{k,\infty}(\R^n,V)$. This we do in two steps. First, we show that the obvious candidate $(\nabla^l u)\chi_{\bar{\Omega}}+(\nabla^l E_\Omega u)\chi_{\R^n\setminus\bar{\Omega}}$ is bounded for all $0\leq l\leq k$. We need only prove this for $l=k$, the other cases being dealt with in Lemma~\ref{lem:localbounds} for $p=\infty$. As before, we first take $Q_0\in\mathcal{W}_3$, where
\begin{align*}
|\nabla^k E_\Omega u|&\leq |\nabla^k\pi_{Q_0^*}u|+\sum_{\emptyset\neq Q_0\cap Q\in\mathcal{W}_3}|\nabla^k(\varphi_Q(\pi_{Q^*} u-\pi_{Q_0^*}u))|\\
&\leq C\left( |\nabla^k\pi_{Q_0^*}u|+\sum_{\emptyset\neq Q_0\cap Q\in\mathcal{W}_3}\|\nabla^k u\|_{\infty,\mathcal{C}(Q_0^*,Q^*)}\right).
\end{align*}
Clearly, $P\mapsto\|\nabla^k P\|_{\infty,Q_0^*}$ is a norm on $\R_d[x]^V/\R_{k-1}[x]$, whereas $P\mapsto\|\nabla^k \Pi P\|_{\infty,Q_0^*}$ is a semi--norm. We therefore get that $\|\nabla^k\pi_{Q_0^*}u\|_{\infty,Q_0^*}\leq C \|\nabla^k\mathcal{P}_{Q_0^*}u\|_{\infty,Q_0^*}\leq C\|\nabla^k u\|_{\infty,Q_0^*}$, where the latter inequality is given by the stability of averaged Taylor polynomials. Now consider the other case, when $Q_0\in\mathcal{W}_2\setminus\mathcal{W}_3$, and recall that then $\ell(Q_0)\geq C$. We have
\begin{align*}
|\nabla^l E_\Omega u|&\leq \sum_{\emptyset\neq Q_0\cap Q\in\mathcal{W}_3}|\nabla^k(\varphi_Q\pi_{Q^*})|\leq C\sum_{\emptyset\neq Q_0\cap Q\in\mathcal{W}_3}\sum_{j=1}^k\ell(Q)^{j-k}|\nabla^j\pi_\Omega u|\\
&\leq C\sum_{\emptyset\neq Q_0\cap Q\in\mathcal{W}_3}\sum_{j=1}^k\ell(Q_0)^{j-k}|\nabla^j\pi_\Omega u|\leq C\sum_{\emptyset\neq Q_0\cap Q\in\mathcal{W}_3}\sum_{j=1}^k|\nabla^j\pi_\Omega u|,
\end{align*}
so we can conclude as in the previous step. 

The second step is to show that $\nabla^l E_\Omega u$ is continuous for $0\leq l < k$. To this end, it suffice to show that
\begin{align*}
\|\nabla^lE_\Omega u-(\nabla^lu)_{Q_0^*}\|_{\infty,Q_0}\rightarrow0\qquad\text{ as }\ell(Q_0)\rightarrow0
\end{align*}
for $Q_0\in\mathcal{W}_3$. Here $(\,\cdot\,)_{S}$ denotes the average with respect to Lebesgue measure on $S$. By the triangle inequality and properties of the partition of unity, we get
\begin{align*}
\|\nabla^lE_\Omega u-(\nabla^lu)_{Q_0^*}\|_{\infty,Q_0}&\leq\left\|\nabla^{l}\sum_{\emptyset\neq Q_0\cap Q\in\mathcal{W}_3}\varphi_Q(\pi_{Q^*}-\pi_{Q_0^*})\right\|_{\infty,Q_0}\\
&+\|\nabla^l\pi_{Q_0^*}u-(\nabla^lu)_{Q_0^*}\|_{\infty,Q_0}=\mathbf{I}+\mathbf{II}.
\end{align*}
By an estimation which is by now routine (see the proof of Lemma~\ref{lem:localbounds}) we have that
\begin{align*}
\mathbf{I}\leq C\ell(Q_0)^{k-l}\sum_{\emptyset\neq Q_0\cap Q\in\mathcal{W}_3}\|\A u\|_{\infty,\cup\mathcal{C}(Q_0^*,Q^*)},
\end{align*}
which tends to zero as $\ell(Q_0)\rightarrow0$ since $k>l$. For the second term, we have by \cite[Lem.~2.1]{Jones} and closeness of $Q_0$ and $Q_0^*$ that
\begin{align*}
\mathbf{II}&\leq C\|\nabla^l\pi_{Q_0^*}u-(\nabla^lu)_{Q_0^*}\|_{\infty,Q_0^*}\leq C\|\nabla^l\pi_{Q_0^*}u-\nabla^lu\|_{\infty,Q_0^*}+C\|\nabla^lu-(\nabla^lu)_{Q_0^*}\|_{\infty,Q_0^*}\\
&\leq C\ell(Q_0)\|\A u\|_{\infty,Q_0^*}+C\ell(Q_0)\|\nabla^{l+1} u\|_{\infty,Q_0^*},
\end{align*}
where the last inequality we used Proposition~\ref{prop:poinc} and the fact that $\nabla^l u$ is Lipschitz as $l<k$.

We next note that by density of smooth functions in $\sobo^{\A,p}(\Omega)$ (Lemma~\ref{lem:density}), it suffices to prove boundedness of the extension for maps in $u\in\hold^\infty(\bar\Omega,V)$. By the argument above, we have that $E_{\Omega}u\in\Vsob^{k,\infty}(\R^n,V)$, so that
\begin{align*}
\nabla^lE_\Omega u&=(\nabla^l u)\chi_{\bar\Omega}+(\nabla^l E_\Omega u)\chi_{\R^n\setminus\bar\Omega}\qquad\text{for } {0\leq l\leq k-1}\\
\A E_\Omega u&=(\A u)\chi_{\bar\Omega}+(\A E_\Omega u)\chi_{\R^n\setminus\bar\Omega}.
\end{align*}
It follows that
\begin{align*}
\|E_{\Omega}u\|_{\Vsob^{\A,p}(\R^n)}\leq \|u\|_{\Vsob^{\A,p}(\Omega)}+\|E_{\Omega}u\|_{\Vsob^{\A,p}(\R^n\setminus\bar\Omega)}\leq C \|u\|_{\sobo^{\A,p}(\Omega)},
\end{align*}
where in the last inequality we used Lemmas~\ref{lem:localbounds} and \ref{lem:sob_variants}. The proof is complete.
\end{proof}

\subsection{Proofs of the main results}
We now begin the proof of Theorem~\ref{thm:main}, by showing that \ref{it:main_a}$\implies$\ref{it:main_b}$\implies$\ref{it:main_c}$\implies$\ref{it:main_a} and \ref{it:main_a}$\implies$\ref{it:main_d}$\implies$\ref{it:main_e}$\implies$\ref{it:main_a}. We first prove that \ref{it:main_a} implies \ref{it:main_b} in the scale of Besov spaces and hereafter establish Theorem~\ref{thm:main_k}; the corresponding statement for Lebesgue spaces $\lebe^{\frac{n}{n-1}}$ for $k=1$ follows in the same way.
\begin{proof}[Proof of Theorem~\ref{thm:main_k} (sufficiency of FDN)] Since $\A$ has FDN, by Theorem~\ref{thm:tools}, $\A$ is elliptic and cancelling and there exists a bounded, linear extension operator $E_{\ball} \colon\sobo^{\A,1}(\ball)\rightarrow\Vsob^{\A,1}(\R^n)$. A close inspection of the proof of Theorem~\ref{thm:extension} reveals that $E_{\ball}$ maps restrictions to the ball $\ball$ of $\hold^\infty(\R^n,V)$--functions into $\hold^\infty_c(\tilde{\ball},V)$ for a larger ball $\tilde{\ball}\Supset\ball$, which depends on $\ball$ only. We write $p:=n/(n-k+s)$ and use H\"older's Inequality to get that
\begin{align*}
\|u\|_{{\besov}_q^{s,p}(\ball,V)}&\leq\|E_{\ball}u\|_{{\besov}^{s,p}_q(\R^n,V)}=\|E_{\ball}u\|_{\lebe^p(\tilde{\ball},V)}+\|E_{\ball}u\|_{\dot{\besov}{_q^{s,p}}(\R^n,V)}\\
&\lesssim\|E_{\ball}u\|_{\lebe^{\frac{n}{n-1}}(\tilde{\ball},V)}+\|E_{\ball}u\|_{\dot{\besov}{_q^{s,p}}(\R^n,V)}\\
&\lesssim\|\nabla^{k-1}E_{\ball}u\|_{\lebe^{\frac{n}{n-1}}(\tilde{\ball},V)}+\|E_{\ball}u\|_{\dot{\besov}{_q^{s,p}}(\R^n,V)}
\end{align*}
where the last estimate follows from Poincar\'e's Inequality with zero boundary values. We conclude by \cite[Thm.~1.3,~Thm.~8.4]{VS} and boundedness of $E_{\ball}$.
\end{proof}
We will complete the proof of Theorem~\ref{thm:main_k} (i.e., show necessity of FDN) at the end of this section. Returning to Theorem~\ref{thm:main}, it is clear that \ref{it:main_b} implies \ref{it:main_c}.
\begin{proof}[Proof of \ref{it:main_c}$\implies$\ref{it:main_a} in Theorem~\ref{thm:main}]
	Suppose that $\sobo^{\A,1}(B)\hookrightarrow\lebe^p(B,V)$, but $\A$ is not $\C$--elliptic. By Lemma~\ref{lem:nec_ell}, we have that $\A$ is elliptic, so there exist pairwise linearly independent $\xi,\eta\in\R^n$, $v,w\in V$ such that $\A[\xi+\imag \eta](v+\imag w)=0$.
	Define an inner product on $\R^n$ such that $\xi,\eta$  are orthonormal and define an orthonormal basis with respect to this new inner product. Say that $Q$ is the change of base matrix. By changing the basis to the new one, we obtain an inclusion $\sobo^{\tilde\A,1}(C)\hookrightarrow\lebe^p(C,V)$, where $C=QB$ is an ellipsoid and $\tilde\A$ is the expression of $\A$ in the new basis. We also write $\tilde{\xi}=Q\xi,\tilde \eta=Q\eta$ and have that $\tilde{\A}[\tilde\xi+\imag \tilde\eta](v+\imag w)=0$.
	 By the computation in the proof of Lemma~\ref{lem:FDNimpliesEC}, we have that 
	 \begin{align}\label{eq:u}
		u(x)=f(x\cdot \tilde\xi+\imag x\cdot \tilde\eta)(v+\imag w)
	 \end{align}
	 	  is $\tilde\A$--free at all points where $f$ is holomorphic (here ``$\,\cdot\,$'' is the inner product in the new basis). Translating $C$ if necessary, we can assume that $\xi$ is parallel with the inward normal at $0\in\partial C$. By flattening the boundary of $C$ near $0$, we can assume that there is a flat neighbourhood of $0$ in $\partial C$, i.e., contained in $\{x\in\R^n\colon x\cdot\tilde\xi=0\}$.
	 
	 Let $f\colon\C\setminus(-\infty,0]\rightarrow\C$ be a branch of $z^{-2/p}$ in \eqref{eq:u}. Since $p>1$, the singularity $|z|^{-2/p}$ is integrable near zero (in $\mathrm{span}\{\tilde\xi,\tilde\eta\}\simeq \C$), from which one easily infers that $u\in\lebe^1(C,V)$. Coupled with the fact that $\tilde{\A}u=0$ in $C$ from the geometrical construction, we have that $u\in\sobo{^{\tilde\A,1}}(C)$. On the other hand, we have that $|u|^p$ is integrable a neighbourhood of $0$ in $C$ if and only if $|z|^{-2}$ is integrable in neighbourhood of $0$ in $\{x\in\mathrm{span}\{\tilde\xi,\tilde\eta\}\colon x\cdot\tilde\xi>0\}$, which is clearly not the case. Hence $u\in\sobo^{\tilde\A,1}(C)$, but $u\notin\lebe^p(C,V)$. The proof is complete.
\end{proof}

To see that \ref{it:main_a} implies \ref{it:main_d}, we prove the following:
\begin{theorem}\label{thm:compactness}
Let $\A$ be as in \eqref{eq:A} with $k=1$. Suppose that $\A$ is $\C$--elliptic. Then $\sobo^{\A,1}(\ball)\hookrightarrow\hookrightarrow\lebe^{q}(\ball,V)$ for all $1\leq q<\frac{n}{n-1}$.
\end{theorem}
The proof of Theorem~\ref{thm:compactness} relies on the Riesz--Kolmogorov criterion  and the following Nikolski\u{\i}--type estimate:
%\textcolor{blue}{I will include a proof of this, precisely, formula (4.1) on Friday, Aug 3. This is more intricate than I had thought before.}
\begin{lemma}[Nikolski\u{\i}--type Estimate]\label{lem:Besov}
Let $\A$ be an elliptic operator of the form \eqref{eq:A}, $k=1$. Fix $R>0$. Then for every $0<s<1$ there exists a constant $c=c(s,R)>0$ such that if $u\in\sobo^{\A,1}(\R^{n})$ vanishes identically outside $\ball(0,R)$, then there holds
\begin{align*}
\int_{\R^{n}}|u(x+y)-u(x)|^{p}\dif x \leq c\|\A u\|_{\lebe^{1}(\ball(0,R),W)}^{p}|y|^{sp}.
\end{align*}
whenever $p<n/(n-1+s)$. 
\end{lemma}
Note that by {Ornstein}'s Non--inequality, $s=1$ is not allowed in the lemma. A more general, sharp version that implies Lemma~\ref{lem:Besov} %showing in addition that ellipticity is also necessary for the estimate, 
can be found in \cite[Prop.~8.22]{VS}. 
\begin{proof}[Proof of Theorem~\ref{thm:compactness}]
	Let $(u_j)_j$ be a bounded subset of $\sobo^{\A,1}(B)$. Employing the extension $E_B$ of Theorem~\ref{thm:extension}, we have that there exists $R>0$ such that the functions $(E_B u_j)_j$ are supported in $B(0,R)$ and are uniformly bounded in $\sobo^{\A,1}(\R^n)$. We record that $\|\A E_Bu_j\|_{\lebe^1(\R^n,W)}\leq c$. It follows by Lemma~\ref{lem:FDNimpliesEC} and the estimate \eqref{eq:VS} that $(E_Bu_j)_j$ is bounded in $\lebe^{n/(n-1)}(B(0,R),V)$, hence in $\lebe^q(B(0,R),V)$ by H\"older's Inequality. 
	
	Let $\varepsilon>0$ and $s\in(0,1)$ be such that $q<n/(n-1+s)$. By Lemma~\ref{lem:Besov}, we have that
	\begin{align*}
		\|E_Bu_j(\cdot+y)-E_B u_j\|_{\lebe^q(\R^n,V)}\leq c_0|y|^s< \varepsilon,
	\end{align*}
	provided that $|y|<(\varepsilon/c_0)^{1/s}$. It follows by the Riesz--Kolmogorov Theorem \cite[Thm.~4.26]{BrezisFA} that the sequence $(E_Bu_j)_j$ has a convergent subsequence which we do not relabel. Since $(E_Bu_j)_j$ converges in $\lebe^q(B(0,R),V)$, it follows that $(u_j)_j$ converges in $\lebe^q(B,V)$. The proof is complete.
\end{proof}
It is clear that \ref{it:main_d}$\implies$\ref{it:main_e}, so that we can now complete the:
\begin{proof}[Proof of Theorem~\ref{thm:main}]
It remains to see that \ref{it:main_e} implies \ref{it:main_a}, which is now a simple consequence of the Equivalence Lemma~\ref{lem:equivalencelemma}. We  choose $E_{1}=\sobo^{\A,1}(\ball)$, $E_{2}=\lebe^{1}(\ball,W)$, $E_{3}=\lebe^{1}(\ball,V)$, and $A:=\A\in\mathscr{L}(\sobo^{\A,1}(\ball),\lebe^{1}(\ball,W))$, whereas $B:=\iota$ is the embedding operator $\iota\colon\sobo^{\A,1}(\ball)\hookrightarrow\hookrightarrow \lebe^{1}(\ball,V)$. It is then clear that $\|u\|_{\sobo^{\A,1}(\ball)}= \|u\|_*$, so the equivalence lemma yields that $\A$ has finite dimensional null--space.
\end{proof}
\begin{proof}[Proof of Theorem~\ref{thm:main_k} (necessity of FDN)]
Assume that the embedding holds. By standard embeddings of Besov spaces, we have that $\sobo^{\A,1}(\ball)\hookrightarrow\sobo^{k-1,p}(\ball,V)$ for some $p>1$. If $k=1$, we use Theorem~\ref{thm:main}, \ref{it:main_c} implies \ref{it:main_a}, to see that $\A$ has FDN. Otherwise, we give the following simple argument: assume that $\A$ is not FDN, so that the maps $u_j(x)=\exp(jx\cdot\xi)v$ lie in $\ker\A$ for some non--zero complex $\xi,v$. We traced this example back to \cite{Smith}, but it was likely known before (cp. \cite[Eq.~(3.2)]{Aronszajn}). The assumed embedding and H\"older's Inequality give
\begin{align*}
j^{k-1}\left(\int_{\ball}|\exp(jx\cdot\xi)|^p\dif x\right)^\frac{1}{p}&\lesssim \|u_j\|_{\sobo^{k-1,p}(\ball,V)}\lesssim\|u_j\|_{\lebe^1(\ball,V)}\\
&\lesssim\left(\int_{\ball}|\exp(jx\cdot\xi)|^p\dif x\right)^\frac{1}{p},
\end{align*}
which leads to a contradiction as $j\rightarrow\infty$. Here constants depend on $\diam\ball,p,n$ only.
\end{proof}
\subsection{On $\bv^\A$ and traces}
While we do not mention this explicitly, by strict density of smooth maps in the space of measures \cite[Thm.~2.8]{BDG}, in the framework of Theorem~\ref{thm:main_k} we have that $\A$ of order $k$ has FDN if and only if 
\begin{align}\label{eq:BVA_emb}
\|u\|_{\besov^{s,\frac{n}{n-k+s}}_q (\ball,V)}\leq C\left(|\A u|(\ball)+\|u\|_{\lebe^1(\ball,V)}\right)\quad\text{ for all }u\in\bv^\A(B).
\end{align}
This in particular shows that $\bv^\A(B)\subset\sobo^{k-1,n/(n-1)}(B,V)$, and it is likely that the inclusion is strictly continuous by an argument similar to \cite[Prop. 3.7]{RS} in the $\bv$--case. We do not pursue this here.

The proof of \eqref{eq:BVA_emb} can be simplified considerably for first order operators, i.e., if $k=1$. In this case, one can use the main result of \cite{BDG} to show that the trivial extension $\bar u$ by zero on $\R^n\setminus\bar B$ of a map $u\in\bv^\A(B)$ does lie in $\bv^\A(\R^n)$:
\begin{align*}
|\A \bar u|(\R^n)=|\A u|(B)+\|\tr u\otimes_\A \nu_{\partial B}\|_{\lebe^1(\partial B,W)}\leq C(|\A u|(B)+\|u\|_{\lebe^1(B,V)}).
\end{align*}
In this case, \textsc{Van Schaftingen}'s estimates in \cite[Sec.~8]{VS} are applicable by Lemma~\ref{lem:FDNimpliesEC} and strict density to obtain
\begin{align*}
|u|_{\besov^{s,\frac{n}{n-1+s}}_q(B) }\leq |\bar u|_{\besov^{s,\frac{n}{n-1+s}}_q (\R^n)}\leq C|\A \bar u|(\R^n)\leq C(|\A u|(B)+\|u\|_{\lebe^1(B,V)}).
\end{align*}
We do not know any generalization of this method for higher order operators. 

Another way to construct simpler extension operators for $\bv^\A(B)$--maps for \emph{first order} FDN operators $\A$ can be achieved by letting $u\in\bv^\A(B)$ and picking a map $v\in\bv(\R^{n}\setminus \bar{B},V)$ such that $\trace(v)=\trace(u)$ $\mathcal{H}^{n-1}$--a.e. on $\partial B$. This can be arranged in a way such that $\|v\|_{\bv(\R^{n}\setminus\overline{B},V)}\leq C\|\tr u\|_{\lebe^{1}(\partial B,V)}$. We then have that 
\begin{align*}
\tilde{E}u=\begin{cases}
u&\text{in }B\\
v&\text{in }\R^n\setminus\bar B
\end{cases}
\end{align*}
lies in $\bv^\A(\R^n)$ and one can check that 
\begin{align*}
\|\tilde Eu\|_{\bv^\A(\R^n)}&\leq \|u\|_{\bv^\A(B)}+\|v\|_{\bv(\R^n\setminus\bar B)}\\
&\leq \|u\|_{\bv^\A(B)}+C\|\tr u\|_{\lebe^1(\partial B,V)}
\leq C\|u\|_{\bv^\A(B)},
\end{align*}
where, in the last inequality, we again used the main result in \cite{BDG}. In this situation, $\tilde{E}$ is never linear as a consequence of \textsc{Peetre}'s foundational work \cite{Peetre} (see also \cite{PelWoj}); in fact, the map $\mathbf{L}\colon \lebe^{1}(\partial B,V)\to\bv(\R^{n}\setminus\bar{B},V)$ cannot be chosen to be linear \emph{and} bounded simultaneously. Compared with this extension operator $\tilde{E}$, the operator $E_{B}$ from Theorem~\ref{thm:tools} is linear. This, however, is in accordance with the former result as $E_{B}$ extends given $\bv^{\A}$--maps defined on $B$ \emph{and not} $\lebe^{1}$--maps defined on $\partial B$.

One can hope to generalise this argument to higher order operators. Indeed, the method in \cite{BDG} implies the fact that there exists a continuous, linear trace operator $\tr\colon\bv^\A(B)\rightarrow\sobo^{k-1,1}(\partial B,V)$, where $k$ is the order of $\A$. One would then hope for the existence of a bounded linear operator, mapping $\sobo^{k-1,1}(\partial B,V)$ to $\bv^k(\R^n\setminus\bar B,V)$, which would enable us to construct an operator such as $\tilde E$ above.
Building on the work of \textsc{Uspenski\u{\i}} \cite{Uspenskii} for $k=2$ and \textsc{Maz'ya} \cite{Mazya}, no such operator exists, as was recently rediscovered and strengthened by \textsc{Mironescu} and \textsc{Russ} in \cite{MR}. They showed that the trace operator on $\bv^k(\R^n_+)$ is continuous onto $\besov^{k-1,1}_1(\R^{n-1})$, which is in general strictly smaller than $\sobo^{k-1,1}(\R^{n-1})$ (see \cite[Rk.~A.1]{BP}). From this point of view one would expect:
\begin{conjecture}
An operator $\A$ as in \eqref{eq:A} with $k\ge2$ has FDN if and only if there exists a continuous, linear, surjective trace operator $\mathrm{Tr}\colon\bv^{\A}(\ball)\rightarrow\besov^{k-1,1}_1(\partial\ball,V)$.
\end{conjecture}
A few remarks are in order. Necessity of FDN can be proved by a modification of the arguments in \cite[Sec.~4.3]{BDG}. Surjectivity is obvious, using \cite[Thm.~1.3-4]{MR} and $\bv^{k}(\ball,V)\hookrightarrow\bv^{\A}(\ball)$. The difficulty stems from proving boundedness (hence, well--definedness) of the trace operator, which cannot be reduced to the situation in \cite{MR} by {Ornstein}'s Non--inequality. By \eqref{eq:BVA_emb}, one can prove a sub--optimal trace inequality
\begin{align}\label{eq:weaktrace}
\bv^{\A}(\ball)\hookrightarrow\besov^{s-\frac{1}{p},p}_q(\partial\ball,V)\quad\text{ for }s\uparrow k\text{, so }p=\frac{n}{n-k+s}\downarrow1,\,q\downarrow1,
\end{align}
using standard trace theory for Besov spaces, but this is of course insufficient to solve the problem. We do not see a straightforward way to merge the techniques in \cite{BDG,MR} and intend to tackle the problem in the future.

%\begin{align*}
%\|u-u_j\|_{\sobo^{\A,p}(\Omega)}\rightarrow0\qquad\text{ as }j\rightarrow\infty.
%\end{align*}


\begin{thebibliography}{99}
\bibitem{ADN1} \textsc{Agmon, S.}, \textsc{Douglis, A.} and \textsc{Nirenberg, L.}, 1959. \emph{Estimates near the boundary for solutions of elliptic partial differential equations satisfying general boundary conditions. I}. Communications on pure and applied mathematics, \textbf{12}(4), pp.623-727.
\bibitem{ADN2} \textsc{Agmon, S.}, \textsc{Douglis, A.} and \textsc{Nirenberg, L.}, 1964. \emph{Estimates near the boundary for solutions of elliptic partial differential equations satisfying general boundary conditions. II}. Communications on pure and applied mathematics, \textbf{17}(1), pp.35-92.
%\bibitem{ADM} \textsc{Ambrosio, L.} and \textsc{Dal Maso, G.}, 1992. \emph{On the relaxation in $\bv(\Omega,\R^m)$ of quasi--convex integrals}. Journal of functional analysis, \textbf{109}(1), pp.76-97.%
\bibitem{Aronszajn} \textsc{Aronszajn, N.}, 1954. \emph{On coercive integro-differential quadratic forms}. Conference on Partial Differential
Equations, University of Kansas, Technical Report No. 14, pp.94-106.
\bibitem{Adolfo} \textsc{Arroyo-Rabasa, A.}, 2017. \emph{Relaxation and optimization for linear--growth convex integral functionals under PDE constraints}. Journal of Functional Analysis, \textbf{273}(7), pp.2388-2427.
\bibitem{BB02} \textsc{Bourgain, J.} and \textsc{Brezis, H.}, 2002. \emph{Sur l'\'equation $\di u=f$}. Comptes Rendus Mathematique, \textbf{334}(11), pp.973-976.
\bibitem{BB04} \textsc{Bourgain, J.} and \textsc{Brezis, H.}, 2003. \emph{On the equation $\di Y=f$ and application to control of phases}. Journal of the American Mathematical Society, \textbf{16}(2), pp.393-426.%
\bibitem{BBCR} \textsc{Bourgain, J.} and \textsc{Brezis, H.}, 2004. \emph{New estimates for the Laplacian, the $\di$--$\curl$, and related Hodge systems}. Comptes Rendus Mathematique, \textbf{338}(7), pp.539-543.
\bibitem{BB07} \textsc{Bourgain, J.} and \textsc{Brezis, H.}, 2007. \emph{New estimates for elliptic equations and Hodge type systems}. Journal of the European Mathematical Society, \textbf{9}(2), pp.277-315.%
\bibitem{BBM1} \textsc{Bourgain, J.}, \textsc{Brezis, H.}, and \textsc{Mironescu, P.}, 2000. \emph{Lifting in Sobolev spaces}. Journal d'analyse math\'ematique, \textbf{80}(1), pp.37-86.
\bibitem{BBM2} \textsc{Bourgain, J.}, \textsc{Brezis, H.}, and \textsc{Mironescu, P.}, 2000. \emph{On the structure of the Sobolev space $H^{1/2}$ with values into the circle}. Comptes Rendus de l'Acad\'emie des Sciences-Series I-Mathematics, \textbf{331}(2), pp.119-124.
\bibitem{BBM3} \textsc{Bourgain, J.}, \textsc{Brezis, H.}, and \textsc{Mironescu, P.}, 2004. \emph{$H^{1/2}$ maps with values into the circle: minimal connections, lifting, and the Ginzburg--Landau equation}. Publications Math\'ematiques de l'IH\'ES, \textbf{99}(1), pp.1-115.
\bibitem{BVS} \textsc{Bousquet, P.} and \textsc{Van Schaftingen, J.}, 2014. \emph{Hardy--Sobolev inequalities for vector fields and canceling linear differential operators}. Indiana University Mathematics Journal, pp.1419-1445.
\bibitem{BDG} \textsc{Breit, D.}, \textsc{Diening, L.} and \textsc{Gmeineder, F.}, 2019. \emph{On the trace operator for functions of bounded $\A$-variation}. arXiv preprint arXiv:1707.06804v2, to appear in Analysis and PDE.%
\bibitem{BrezisFA} \textsc{Brezis, H.}, 2010. \emph{Functional analysis, Sobolev spaces and partial differential equations}. Springer Science \& Business Media.%
\bibitem{BP} \textsc{Brezis, H.} and \textsc{Ponce, A.C.}, 2008. \emph{Kato's inequality up to the boundary}. Communications in Contemporary Mathematics, \textbf{10}(06), pp.1217-1241.%
\bibitem{BrezisVS} \textsc{Brezis, H.} and \textsc{Van Schaftingen, J.}, 2007. \emph{Boundary estimates for elliptic systems with $\lebe^1$--data}. Calculus of Variations and Partial Differential Equations, \textbf{30}(3), pp.369-388.
\bibitem{CZ} \textsc{Calder\'on, A.P.} and \textsc{Zygmund, A.}, 1952. \emph{On the existence of certain singular integrals}.
Acta Mathematica, \textbf{88}(1), pp.85-139.
%\bibitem{CFM}  \textsc{Conti, S.}, \textsc{Faraco, D.}, and \textsc{Maggi, F.}, 2005. \emph{A new approach to counterexamples to $\lebe^1$ estimates: Korn's inequality, geometric rigidity, and regularity for gradients of separately convex functions.} Archive for rational mechanics and analysis, \textbf{175}(2), pp.287-300.
\bibitem{devoresharpley} \textsc{DeVore, R.A.} and \textsc{Sharpley, R.C.}, 1993. \emph{Besov spaces on domains in $\R^d$}. Transactions of the American Mathematical Society, \textbf{335}(2), pp.843-864.%
\bibitem{EG} \textsc{Evans, L.C.} and \textsc{Gariepy, R.F.}, 1992. \emph{Measure Theory and Fine Properties of Functions}. CRC Press.
\bibitem{Federer} \textsc{Federer, H.}, 2014. \emph{Geometric measure theory}. Springer.%
\bibitem{FM} \textsc{Fonseca, I.} and \textsc{M\"uller, S.}, 1999. \emph{$\mathcal{A}$--Quasiconvexity, Lower Semicontinuity, and Young Measures}. SIAM journal on mathematical analysis, \textbf{30}(6), pp.1355-1390.%
\bibitem{Fuchs1} \textsc{Fuchs, M.}, 2011. \emph{An estimate for the distance of a complex valued Sobolev function defined on the unit disc to the class of holomorphic functions}. Journal of Applied Analysis, \textbf{17}(1), pp.131-135.%
\bibitem{FS} \textsc{Fuchs, M.} and \textsc{Seregin, G.}, 2000. \emph{Variational methods for problems from plasticity theory and for generalized Newtonian fluids}. Springer Science \& Business Media.%
\bibitem{Gag} \textsc{Gagliardo, E.}, 1958. \emph{Propriet\`a di alcune classi di funzioni in pi\` variabili}. Ricerche mat,
\textbf{7}(1), pp.102-137.
%\bibitem{GMS} \textsc{Giaquinta, M.}, \textsc{Modica, G.}, and \textsc{Sou\v{c}ek, J.}, 1979. \emph{Functionals with linear growth in the calculus of variations, I}. Commentationes Mathematicae Universitatis Carolinae, \textbf{20}(1), pp.143-156.%
\bibitem{GT} \textsc{Gilbarg, D.} and \textsc{Trudinger, N.S.}, 2015. \emph{Elliptic partial differential equations of second order.} Springer.%
\bibitem{GR} \textsc{Gmeineder, F.} and \textsc{Rai\cb{t}\u{a}, B.}, 2019:  \emph{On critical $\lebe^{p}$--differentiability of $\bd$--maps}. Revista Matem\'atica Iberoamericana, online first, doi 10.4171/rmi/1111 (in press).
\bibitem{HormanderBdry} \textsc{H\"ormander, L.}, 1966. \emph{Pseudo--differential operators and non--elliptic boundary problems}. Annals of Mathematics, pp.129-209.
\bibitem{HormanderIII} \textsc{H\"ormander, L.}, 1985. \emph{The Analysis of Linear Partial Differential Operators. III. Pseudodifferential
Operators}. Springer, Berlin.
\bibitem{Jones} \textsc{Jones, P.W.}, 1981. \emph{Quasiconformal mappings and extendability of functions in Sobolev spaces}. Acta Mathematica, \textbf{147}(1), pp.71-88.%
\bibitem{Kalamajska} \textsc{Ka{\l}amajska, A.}, 1994: \emph{Pointwise multiplicative inequalities and Nirenberg type estimates in weighted Sobolev spaces}. Studia Math, \textbf{108}(3), pp.275--290.%
\bibitem{Kalamajska94} \textsc{Ka{\l}amajska, A.}, 1993. \emph{Coercive inequalities on weighted Sobolev spaces}. Colloquium Mathematicae, Vol. \textbf{66}(2), pp.309-318.%
\bibitem{KirKri} \textsc{Kirchheim, B.} and \textsc{Kristensen, J.}, 2016. \emph{On rank one convex functions that are homogeneous of degree one}. Archive for Rational Mechanics and Analysis, \textbf{221}(1), pp.527-558.%
\bibitem{Lopa} \textsc{Lopatinski\u{i}, Y.B.}, 1953. \emph{On a method of reducing boundary problems for a system of differential equations of elliptic type to regular integral equations}. Ukrain. Mat. \v{Z}, \textbf{5}, pp.123-151.
\bibitem{Mazya} \textsc{Maz'ya, V.}, 2013. \emph{Sobolev spaces}. Springer.%
\bibitem{Milton} \textsc{Milton, G.W.}, 2002. \emph{The theory of composites}. Cambridge University Press.%
\bibitem{MR} \textsc{Mironescu, P.} and \textsc{Russ, E.}, 2015. \emph{Traces of weighted Sobolev spaces. Old and new}. Nonlinear Analysis: Theory, Methods \& Applications, \textbf{119}, pp.354-381.%
\bibitem{Nir} \textsc{Nirenberg, L.}, 1959. \emph{On elliptic partial differential equations}. Annali della Scuola Normale
Superiore di Pisa--Classe di Scienze, \textbf{13}(2), pp.115-162.
\bibitem{Ornstein} \textsc{Ornstein, D.}, 1962. \emph{A non-inequality for differential operators in the $\lebe^1$ norm}. Archive for Rational Mechanics and Analysis, \textbf{11}(1), pp.40-49.%
\bibitem{Peetre} \textsc{Peetre, J.}, 1979. \emph{A counterexample connected with Gagliardo’s trace theorem}. Comment. Math. Special Issue, \textbf{2}, pp.277-282.
\bibitem{PelWoj} \textsc{Pe\l czy\'nski, A.} and \textsc{Wojciechowski, M.}, 2003. \emph{Spaces of functions with bounded variation and Sobolev spaces without local unconditional structure}. Journal f\"ur die Reine und Angewandte Mathematik, pp.109-158.
\bibitem{Reshet} \textsc{Reshetnyak, Y.G.}, 1970. \emph{Estimates for certain differential operators with finite-dimensional kernel}. Siberian Mathematical Journal, \textbf{11}(2), pp.315-326.%
\bibitem{RS} \textsc{Rindler, F.} and \textsc{Shaw, G.}, 2015. \emph{Strictly continuous extension of functionals with linear growth to the space $\bv$}. The Quarterly Journal of Mathematics, \textbf{66}(3), pp.953-978.
\bibitem{Smith0} \textsc{Smith, K.T.}, 1961. \emph{Inequalities for formally positive integro-differential forms}. Bulletin of the American Mathematical Society, \textbf{67}(4), pp.368-370.
\bibitem{Smith} \textsc{Smith, K.T.}, 1970. \emph{Formulas to represent functions by their derivatives}. Mathematische Annalen, \textbf{188}(1), pp.53-77.%
\bibitem{Sobolev} \textsc{Sobolev, S.L.}, 1938. \emph{On a theorem of functional analysis}. Mat. Sbornik, \textbf{4}, pp.471-497.
\bibitem{Stein} \textsc{Stein, E.M.}, 2016. \emph{Singular integrals and differentiability properties of functions}. Princeton university press.%
\bibitem{ST} \textsc{Strang, G.} and \textsc{Temam, R.}, 1980. \emph{Functions of bounded deformation}. Archive for Rational Mechanics and Analysis, \textbf{75}(1), pp.7-21.%
\bibitem{Strauss} \textsc{Strauss, M.J.}, 1973. \emph{Variations of Korn's and Sobolev's inequalities}. Proceedings of Symposia in Pure Mathematics \textbf{23}, D. Spencer (ed.), American Mathematical Society, pp.207-214.%
\bibitem{Suquet} \textsc{Suquet, P.M.}, 1979. \emph{Un espace fonctionnel pour les \'equations de la plasticit\'e}. Annales de la Facult\'e des sciences de Toulouse: Math\'ematiques \textbf{1}(1), pp.77-87.%
\bibitem{Tartar} \textsc{Tartar, L.}, 2007. \emph{An introduction to Sobolev spaces and interpolation spaces (Vol. 3)}. Springer Science \& Business Media.%
\bibitem{Uspenskii} \textsc{Uspenski\u{\i}, S.V.}, 1961. \emph{Imbedding theorems for classes with weights}. Trudy Matematicheskogo Instituta imeni VA Steklova, \textbf{60}, pp.282-303.%
\bibitem{VS-1} \textsc{Van Schaftingen, J.}, 2004. \emph{A simple proof of an inequality of Bourgain, Brezis and Mironescu}. Comptes Rendus Mathematique, \textbf{338}(1), pp.23-26.
\bibitem{VS0} \textsc{Van Schaftingen, J.}, 2004. \emph{Estimates for $\lebe^1$--vector fields}. Comptes Rendus Mathematique, \textbf{339}(3), pp.181-186.
\bibitem{VS1} \textsc{Van Schaftingen, J.}, 2008. \emph{Estimates for $\lebe^1$--vector fields under higher--order differential conditions}. Journal of the European Mathematical Society, \textbf{10}(4), pp.867-882.
\bibitem{VS4} \textsc{Van Schaftingen, J.}, 2010. \emph{Limiting fractional and Lorentz space estimates of differential forms}. Proceedings of the American Mathematical Society, \textbf{138}(1), pp.235-240.
\bibitem{VS} \textsc{Van Schaftingen, J.}, 2013. \emph{Limiting Sobolev inequalities for vector fields and canceling linear differential operators}. Journal of the European Mathematical Society, \textbf{15}(3), pp.877-921.
\bibitem{VS3} \textsc{Van Schaftingen, J.}, 2014. \emph{Limiting \textsc{Bourgain--Brezis} estimates for systems of linear differential equations: Theme and variations}. Journal of Fixed Point Theory and Applications, \textbf{15}(2), pp.273-297.
\bibitem{Whitney} \textsc{Whitney, H.}, 1934. \emph{Analytic extensions of differentiable functions defined in closed sets}. Transactions of the American Mathematical Society, \textbf{36}(1), pp.63-89.%
\bibitem{ZZ} \textsc{Zhao, R.} and \textsc{Zhu, K.}, 2008. \emph{Theory of Bergman spaces in the unit ball of $\C^n$}. M\'emoire de la Soci\'et\'e math\'ematique de France, (115).%

\end{thebibliography}
\end{document}